\documentclass[11pt]{article}
\usepackage{amsmath,amsthm,amssymb}
\usepackage[usenames,dvipsnames]{xcolor}
\usepackage{enumerate}
\usepackage{graphicx}
\usepackage{cite}
\usepackage{comment}
\usepackage{oands}
\usepackage{tikz}
\usepackage{changepage}
\usepackage{bbm}
\usepackage{mathtools}
\usepackage[margin=1in]{geometry}
\usepackage{setspace}
\usepackage[pagewise,mathlines]{lineno}
\usepackage{appendix}
\usepackage{stmaryrd} 
\usepackage{multicol}
\usepackage{microtype}
\usepackage[colorinlistoftodos]{todonotes}
\usepackage{enumitem}

\newcommand{\inputtitle}
{Anomalous diffusion of random walk on random planar maps}

\usepackage[pdftitle={\inputtitle},
  pdfauthor={Ewain Gwynne, Tom Hutchcroft},
colorlinks=true,linkcolor=NavyBlue,urlcolor=RoyalBlue,citecolor=PineGreen,bookmarks=true,bookmarksopen=true,bookmarksopenlevel=2,unicode=true,linktocpage]{hyperref}

\theoremstyle{plain}
\newtheorem{thm}{Theorem}[section]
\newtheorem{cor}[thm]{Corollary}
\newtheorem{lem}[thm]{Lemma}
\newtheorem{prop}[thm]{Proposition}

\def\@rst #1 #2other{#1}
\newcommand\MR[1]{\relax\ifhmode\unskip\spacefactor3000 \space\fi
  \MRhref{\expandafter\@rst #1 other}{#1}}
\newcommand{\MRhref}[2]{\href{http://www.ams.org/mathscinet-getitem?mr=#1}{MR#2}}

\theoremstyle{definition}
\newtheorem{defn}[thm]{Definition}
\newtheorem{remark}[thm]{Remark}

\numberwithin{equation}{section}

\newcommand{\dsb}{\begin{adjustwidth}{2.5em}{0pt}
\begin{footnotesize}}
\newcommand{\dse}{\end{footnotesize}
\end{adjustwidth}}

\newcommand{\ssb}{\begin{adjustwidth}{2.5em}{0pt}}
\newcommand{\sse}{\end{adjustwidth}}

\newcommand{\aryb}{\begin{eqnarray*}}
\newcommand{\arye}{\end{eqnarray*}}
\def\alb#1\ale{\begin{align*}#1\end{align*}}
\def\allb#1\alle{\begin{align}#1\end{align}}
\newcommand{\eqb}{\begin{equation}}
\newcommand{\eqe}{\end{equation}}
\newcommand{\eqbn}{\begin{equation*}}
\newcommand{\eqen}{\end{equation*}}

\newcommand{\BB}{\mathbbm}
\newcommand{\ol}{\overline}
\newcommand{\ul}{\underline}
\newcommand{\op}{\operatorname}

\newcommand{\eqD}{\overset{d}{=}}
\newcommand{\ep}{\epsilon}
\newcommand{\rta}{\rightarrow}

\newcommand{\wt}{\widetilde}
\newcommand{\wh}{\widehat} 
\newcommand{\mcl}{\mathcal}

\newcommand{\bdy}{\partial}
\newcommand{\rng}{\mathring}

\let\originalleft\left
\let\originalright\right
\renewcommand{\left}{\mathopen{}\mathclose\bgroup\originalleft}
\renewcommand{\right}{\aftergroup\egroup\originalright}

\title{\inputtitle}

\date{  }
\author{
\begin{tabular}{c} Ewain Gwynne\\[-5pt]\small Cambridge \end{tabular}
\begin{tabular}{c} Tom Hutchcroft\\[-5pt]\small Cambridge \end{tabular} 
}

\begin{document}

\maketitle
 
\begin{abstract}  
We prove that the simple random walk on the uniform infinite planar triangulation (UIPT) typically travels graph distance at most $n^{1/4 + o_n(1)}$ in $n$ units of time. Together with the complementary lower bound proven by Gwynne and Miller (2017) this shows that the typical graph distance displacement of the walk after $n$ steps is $n^{1/4 + o_n(1)}$, as conjectured by Benjamini and Curien (2013). More generally, we show that the simple random walks on a certain family of random planar maps in the $\gamma$-Liouville quantum gravity (LQG) universality class for $\gamma\in (0,2)$---including spanning tree-weighted maps, bipolar-oriented maps, and mated-CRT maps---typically travels graph distance $n^{1/d_\gamma + o_n(1)}$ in $n$ units of time, where $d_\gamma$ is the growth exponent for the volume of a metric ball on the map, which was shown to exist and depend only on $\gamma$ by Ding and Gwynne (2018). Since $d_\gamma > 2$, this shows that the simple random walk on each of these maps is subdiffusive. 

Our proofs are based on an embedding of the random planar maps under consideration into $\mathbb C$ wherein graph distance balls can be compared to Euclidean balls modulo subpolynomial errors. This embedding arises from a coupling of the given random planar map with a mated-CRT map together with the relationship of the latter map to SLE-decorated LQG. 
\end{abstract}

\tableofcontents

\section{Introduction}
\label{sec-intro}

\subsection{Overview}
\label{sec-overview}
  
It is a consequence of the central limit theorem that simple random walk on Euclidean lattices is \emph{diffusive}, meaning that the end-to-end displacement of an $n$-step simple random walk is of order $\sqrt{n}$ with high probability as $n\to\infty$. In the 1980s, physicists including Alexander and Orbach \cite{alexander1982density} and Rammal and Toulouse \cite{rammal1983random} observed through numerical experiment that, in contrast, random walk on many natural \emph{fractal} graphs, such as those arising in the context of critical disordered systems, is \emph{subdiffusive}, i.e., travels substantially slower than it does on Euclidean lattices. More precisely, they predicted that for many such graphs there exists $\beta>2$ such that the end-to-end displacement of an $n$-step simple random walk is typically\footnote{There are several potentially inequivalent ways to define $\beta$ formally, and we leave this deliberately vague for the purposes of this introduction} of order $n^{1/\beta+o(1)}$ rather than $\sqrt{n}$. This phenomenon is known as \emph{anomalous diffusion}.
 
The first rigorous work on anomalous diffusion was carried out by Kesten \cite{MR871905} who proved that $\beta=3$ for the incipient infinite cluster of critical Bernoulli bond percolation on regular trees of degree at least three, and that random walk on the incipient infinite cluster of $\mathbb{Z}^2$ is subdiffusive with respect to the Euclidean metric (very recently, it was shown by Ganguly and Lee~\cite{gl-subdiffusive} that the walk is also subdiffusive w.r.t.\ the intrinsic metric). Since then, a powerful and general methodology has been developed to analyze anomalous diffusion on \emph{strongly recurrent} graphs, i.e., graphs for which the effective resistance between two points is polynomially large in the distance between them. Highlights of this literature include the work of Barlow and Bass \cite{MR1802425,MR1701339}, Barlow, J\'arai, Kumagai, and Slade \cite{BJKS08}, and Kozma and Nachmias \cite{MR2551766}. A  detailed overview is given in \cite{MR3156983}.  
Beyond the strongly recurrent regime, however, progress has been slow and general techniques are lacking.  

In this paper, we analyze the anomalous diffusion of random walks on \emph{random planar maps}, a class of random fractal objects that have been of central interest in probability theory over the last two decades.
Recall that a \emph{planar map} is a graph embedded in the plane in such a way that no two edges cross, viewed modulo orientation-preserving homeomorphisms. A map is a \emph{triangulation} if each of its faces has three sides. Besides their intrinsic combinatorial interest, the study of random planar maps is motivated by their interpretation as discretizations of continuum random surfaces known as \emph{$\gamma$-Liouville quantum gravity (LQG) surfaces}, where  $\gamma \in (0,2]$ is a parameter describing different universality classes of random map models. 
 LQG surfaces with  $\gamma=\sqrt{8/3}$ arise as scaling limits of uniform random planar maps, while other values of $\gamma$ arise as the scaling limits of random planar maps sampled with probability proportional to the partition function of some appropriately chosen statistical mechanics model. 
The theory of LQG and its connection to random planar maps originated in the physics literature with the work of Polyakov~\cite{polyakov-qg1} and was formulated mathematically in~\cite{shef-kpz} (see also~\cite{rhodes-vargas-review,berestycki-gmt-elementary} for surveys of a closely related theory called \emph{Gaussian multiplicative chaos}, which was initiated by Kahane~\cite{kahane}). Random planar maps are not strongly recurrent (rather, effective resistances are expected to grow logarithmically), and our techniques are highly specific to random map models falling into a $\gamma$-LQG universality class for some $\gamma \in (0,2)$. 

We will be primarily interested in \emph{infinite} random planar maps, which arise as the local limits of finite random planar maps with a uniform random root vertex with respect to the Benjamini-Schramm local topology \cite{benjamini-schramm-topology}. One of the most important infinite random planar maps is the uniform infinite planar triangulation (UIPT), first constructed by Angel and Schramm~\cite{angel-schramm-uipt},  which is the local limit of uniform random triangulations of the sphere as the number of triangles tends to $\infty$. 
Strictly speaking, the UIPT comes in three varieties, known as type I, II, and III, according to whether loops or multiple edges are allowed. To avoid unnecessary technicalities, we will work exclusively in the type II case, in which multiple edges are allowed but self-loops are not. 

The \emph{metric} properties of the UIPT and other uniform random maps have been firmly understood for some time now. In particular, Angel~\cite{angel-peeling} established that the volume of a graph distance ball of radius $r$ in the UIPT grows like $r^4$, and it is known that the (type I) UIPT converges under rescaling to a continuum random surface known as the \emph{Brownian plane} \cite{curien-legall-plane,1604.06622}, which also admits a direct and tractable description as a random quotient of the infinite continuum random tree. Similarly, large finite uniform random triangulations are known to converge under rescaling to a well-understood continuum random surface known as the \emph{Brownian map}~\cite{legall-uniqueness,miermont-brownian-map} (see~\cite{ab-simple,aw-core} for the case of type II and III triangulations).  

The understanding of the \emph{spectral} properties of the UIPT is much less advanced, although a candidate for the scaling limit of random walk on the UIPT, namely \emph{Liouville Brownian motion}, has been constructed~\cite{berestycki-lbm,grv-lbm} and is now reasonably well understood. Important early contributions were made by Benjamini and Curien~\cite{benjamini-curien-uipq-walk}, who proved that random walk on the UIPT is subdiffusive, and by Gurel-Gurevich and Nachmias~\cite{gn-recurrence}, who proved that the random walk on the UIPT is recurrent.
Benjamini and Curien proved furthermore that $\beta\geq 3$ if it exists, and conjectured that $\beta=4$ \cite[Conjecture 1]{benjamini-curien-uipq-walk}.
 Alternative proofs of all of these results, using methods closer  to those of the present paper, were recently obtained by Lee~\cite{lee-conformal-growth,lee-uniformizing}. Very recently, Curien and Marzouk~\cite{cm-uipt-walk} have built upon the approach of \cite{benjamini-curien-uipq-walk} to prove the slightly improved bound $\beta \geq 3 +\ep$ for an explicit $\ep \approx 0.03$.
Further recent works have studied the spectral properties of the random walk on causal dynamical triangulations~\cite{chn-causal} and on $\BB Z^2$ weighted by the exponential of a discrete Gaussian free field~\cite{bdg-lqg-rw}, both of which are indirectly related to the models considered in this paper. Finally, recent work of Murugan \cite{1803.11296} has studied anomalous diffusion on certain deterministic fractal surfaces defined via substitution tilings, showing that $\beta$ is equal to the volume growth dimension for several such examples. However, his methods require strong regularity hypotheses on the graph and do not appear to be applicable to random maps such as the UIPT.

In this paper, we prove the conjecture of Benjamini and Curien in the case of the type II UIPT.
Our techniques also allow us to prove analogous theorems for several other random map models, see Section \ref{sec-main-results}. We use $\op{dist}^G(x,y)$ to denote the graph distance between two vertices $x$ and $y$ in the graph $G$.
  
\begin{thm} \label{thm-walk-speed-uipt}
Let $(M,\BB v)$ be the uniform infinite planar triangulation of type II, and let $X$ be a simple random walk on $M$ started at $\BB v$. Almost surely,
\begin{align} 
 \label{eqn-walk-speed-uipt}
 \lim_{n\rta\infty} \frac{\log \max_{1\leq j \leq n } \op{dist}^M(\BB v, X_j) }{\log n} &= \frac14. 
 \end{align} 
\end{thm}

The first author and Miller \cite{gm-spec-dim} proved the lower bound for graph distance displacement required for Theorem \ref{thm-walk-speed-uipt} (i.e., the inequality $\beta \leq 4$).
So, to prove Theorem \ref{thm-walk-speed-uipt} it suffices for us to prove the upper bound of \eqref{eqn-walk-speed-uipt} (i.e., the inequality $\beta\geq 4$). 
We emphasize that the proof of this upper bound does not use the results of~\cite{gm-spec-dim} or the related works~\cite{gms-tutte,gms-random-walk,gms-harmonic}. 

The central idea behind the techniques of both this paper and \cite{gm-spec-dim} is that a much more refined study of the random walk on the UIPT is possible once one takes the mating-of-trees perspective on random planar maps and SLE-decorated LQG. In particular, both papers rely heavily on the deep work of of Duplantier, Miller, and Sheffield~\cite{wedges}, which rigorously established for the first time a weak form of the long-conjectured convergence of random planar maps toward LQG. This was done by encoding SLE-decorated LQG in terms of a correlated two-dimensional Brownian motion, an encoding that will be of central importance in this paper. The type of convergence considered in~\cite{wedges} is called \emph{peanosphere convergence} and is proven for various types of random planar maps in~\cite{mullin-maps,bernardi-maps,bernardi-dfs-bijection,shef-burger,kmsw-bipolar,gkmw-burger,lsw-schnyder-wood,bhs-site-perc}.
 
Our proof can very briefly be summarized as follows; a more detailed overview is given in~\ref{sec-embedding-intro}. First, we use the mating-of-trees perspective on the theory of SLE and Liouville quantum gravity (in particular, the results of~\cite{wedges,ghs-map-dist}) to define an embedding of a large finite submap of the UIPT into $\mathbb{C}$ with certain desirable geometric properties.  
More precisely, this embedding is obtained by using a bijective encoding of the UIPT by a two-dimensional random walk~\cite{bernardi-dfs-bijection,bhs-site-perc} and a KMT-type coupling theorem~\cite{zaitsev-kmt} to couple the UIPT with a \emph{mated-CRT map}, a random planar map constructed from a correlated two-sided two-dimensional Brownian motion, in such a way that (large subgraphs of) the two maps differ only by a rough isometry. 
The paper~\cite{wedges} gives a natural embedding of the mated-CRT map into $\BB C$ which comes from the encoding of SLE-decorated LQG in terms of correlated two-dimensional Brownian motion. 
We use this to obtain an embedding of the UIPT into $\BB C$ by composing the above embedding of the mated-CRT map with the rough isometry from the UIPT to the mated-CRT map.
  
In~\cite{gm-spec-dim}, this same coupling between the UIPT and the mated-CRT map was used to prove that the effective resistance between the root vertex and the boundary of the ball of radius $r$ in the UIPT grows at most polylogarithmically in $r$. However, while it may be possible in principle to prove $\beta \geq 4$ using electrical techniques, doing so appears to require matching upper and lower bounds for effective resistances on the UIPT differing by at most a constant order multiplicative factor. Such estimates seem to be out of reach of present techniques, which produce polylogarithmic multiplicative errors.  

Instead, we will follow an approach inspired by that of~\cite{lee-conformal-growth} to bound the displacement of random walk on certain random planar maps, which is based on applying the theory of \emph{Markov-type inequalities}  to the above embedding of the UIPT. In particular, we will apply the Markov-type inequality for weighted planar graph metrics due to Ding, Lee, and Peres~\cite{dlp-markov-type} to a weighted metric on the UIPT which approximates the Euclidean distance under the embedding. Background on Markov-type inequalities is given in Section~\ref{sec-markov-type}.  

Note that while Markov-type inequalities are typically used to prove diffusive upper bounds on the walk, our application is more subtle than this, and does \emph{not} prove a diffusive upper bound for the random walk with respect to the Euclidean metric in the embedding (c.f.\ Theorem~\ref{thm-euc-displacement}). Instead, we prove bounds that yield useful information only when $n$ takes values in certain intermediate scales as compared to the natural scale of the embedding. 
The eventual $n^{1/4+o(1)}$ bound on the graph-distance displacement is obtained by taking $n$ to be ``nearly macroscopic" and using that the typical graph distance diameter of a Euclidean ball under our embedding can be estimated modulo subpolynomial errors due to the results of~\cite{dg-lqg-dim}.\footnote{In the case of the UIPT, we expect that this estimate can alternatively be established using estimates for type-II triangulations with simple boundary which come from~\cite{aasw-type2}, but we will not carry this out here.}

Our proof of Theorem~\ref{thm-walk-speed-uipt} does not apply in the case of the UIPQ, the reason being that we do not have a mating-of-trees type bijection which encodes the UIPQ by means of a random walk with i.i.d.\ increments (see Section~\ref{sec-mated-crt-map}). 
 
\medskip \noindent
\textbf{Acknowledgments.}  We thank two anonymous referees for helpful comments on an earlier version of this manuscript. We thank Marie Albenque, Nina Holden, Jason Miller, Asaf Nachmias, and Xin Sun for helpful discussions. We thank Asaf in particular for bringing the maximal versions of the Markov-type inequalities to our attention. This work was initiated during a visit by TH to MIT, whom he thanks for their hospitality.

\subsection{Mated-CRT map background}
\label{sec-mated-crt-map}

A key tool in the proofs of our main results is the theory of mated-CRT maps, which provide a bridge between combinatorial random planar map models (like the UIPT) and the continuum theory of SLE/LQG.  
Let $\gamma \in (0,2)$ and let $Z = (L,R)$ be a two-sided, two-dimensional Brownian motion with variances and covariances
\eqb \label{eqn-bm-cov}
\op{Var}(L_t) = \op{Var}(R_t) = |t| \quad \op{and} \quad \op{Cov}(L_t, R_t) = -\cos(\pi\gamma^2/4) |t| ,\quad \forall t \in \BB R.
\eqe  
Note that this correlation ranges from $-1$ to 1 as $\gamma$ ranges from 0 to 2. 
The $\gamma$-mated CRT map is a discretized mating of the continuum random trees (CRT's) associated with $L$ and $R$. 
Precisely, for $\ep > 0$ the \emph{$\gamma$-mated-CRT map with spacing $\ep$} is the graph $\mcl G^\ep$ with vertex set $ \ep \BB Z$, with two vertices $x_1,x_2\in  \ep \BB Z$ with $x_1<x_2$ connected by an edge if and only if
\allb \label{eqn-inf-adjacency}
&\left( \inf_{t\in [x_1- \ep , x_1]} L_t \right) \vee \left( \inf_{t\in [x_2- \ep , x_2]} L_t \right) \leq \inf_{t\in [x_1  , x_2 -\ep]} L_t \quad \op{or}\quad \notag\\
&\qquad \qquad \left( \inf_{t\in [x_1- \ep , x_1]} R_t \right) \vee \left( \inf_{t\in [x_2- \ep , x_2]} R_t \right) \leq \inf_{t\in [x_1  , x_2 -\ep]} R_t .
\alle
If both conditions in~\eqref{eqn-inf-adjacency} hold and $|x_1-x_2|>\ep$, then there are two edges between $x_1$ and $x_2$. By Brownian scaling, the law of $\mcl G^\ep$ (as a graph) does not depend on $\ep$, but it is convenient to distinguish graphs with different values of $\ep$ since these graphs have different natural embeddings into $\BB C$ (see the discussion just below). Figure~\ref{fig-mated-crt-map} provides a geometric description of the adjacency condition~\eqref{eqn-inf-adjacency} and an explanation of how to put a planar map structure on $\mcl G^\ep$ under which it is a triangulation. 
 
\begin{figure}[t!]
 \begin{center}
\includegraphics[scale=.65]{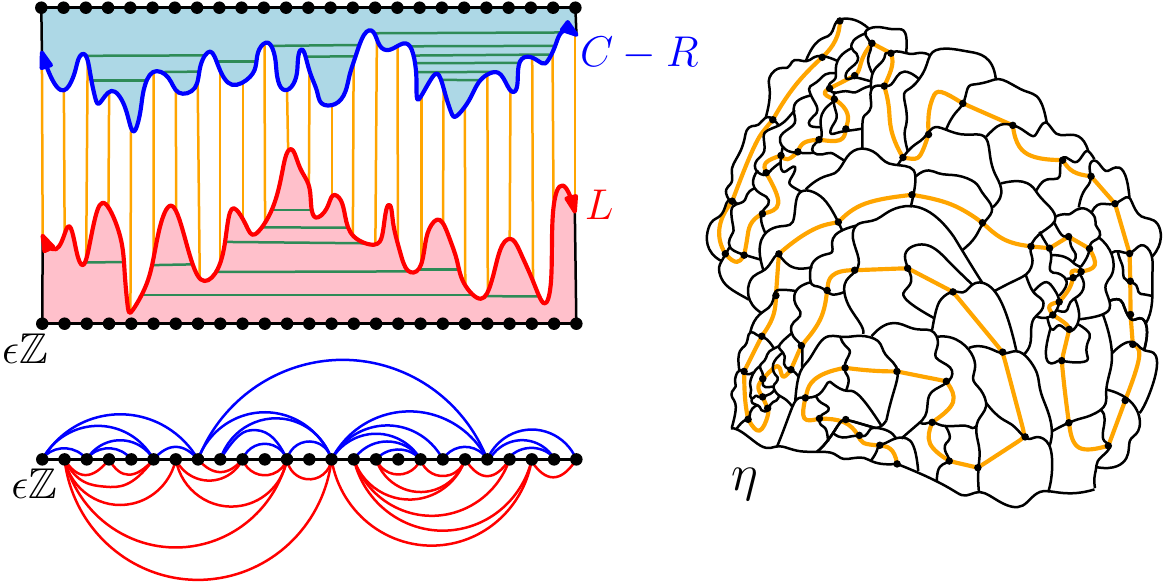}
\vspace{-0.01\textheight}
\caption{\textbf{Top Left:} To construct the mated-CRT map $\mcl G^\ep$ geometrically, one can draw the graph of $L$ (red) and the graph of $C-R$ (blue) for some large constant $C > 0$ chosen so that the parts of the graphs over some time interval of interest do not intersect. One then divides the region between the graphs into vertical strips (boundaries shown in orange) and identifies each strip with the horizontal coordinate $x\in \ep \BB Z$ of its rightmost point. Vertices $x_1,x_2\in \ep \BB Z$ are connected by an edge if and only if the corresponding strips are connected by a horizontal line segment which lies under the graph of $L$ or above the graph of $C-R$. One such segment is shown in green in the figure for each pair of vertices for which this latter condition holds.
\textbf{Bottom Left:} One can draw the graph $\mcl G^\ep $ in the plane by connecting two vertices $x_1,x_2 \in  \ep\BB Z$ by an arc above (resp.\ below) the real line if the corresponding strips are connected by a horizontal segment above (resp.\ below) the graph of $C-R$ (resp.\ $L$), and connecting each pair of consecutive vertices of $ \ep\BB Z$ by an edge. This gives $\mcl G^\ep$ a planar map structure under which it is a triangulation. 
\textbf{Right:} The mated-CRT map can be realized as the adjacency graph of \emph{cells} $\eta([x-\ep,x])$ for $x\in \ep\BB Z$, where $\eta$ is a space-filling SLE$_\kappa$ for $\kappa=16/\gamma^2$ parametrized by $\gamma$-LQG mass with respect to an independent $\gamma$-LQG surface. Here, the cells are outlined in black and the order in which they are hit by the curve is shown in orange. 
Note that the three pictures do not correspond to the same mated-CRT map realization. 
Similar figures have appeared in~\cite{ghs-map-dist,gm-spec-dim,dg-lqg-dim}. 
}\label{fig-mated-crt-map}
\end{center}
\vspace{-1em}
\end{figure} 

The above definition of the mated-CRT map is a continuum analogue of so-called \emph{mating-of-trees} bijections for various infinite-volume combinatorial random planar map models. Such bijections encode a random planar map decorated by a statistical mechanics model via a two-sided two-dimensional random walk $\mcl Z = (\mcl L ,\mcl R) : \BB Z\rta\BB Z^2$, with step distribution depending on the model. For example, for the UIPT, the step distribution is uniform on $\{(1,0),(0,1),(-1,-1)\}$. 
The precise details of the bijection are slightly different for different models, but
 the overall form of the bijection is universal:
In each case, the statistical mechanics model gives rise to a correspondence (not necessarily bijective) between vertices of the map and $\BB Z$ and the condition (in terms of the encoding walk) for two vertices to be adjacent  is a discrete analogue of~\eqref{eqn-inf-adjacency}. The correlation of the coordinates of the walk for planar map models in the $\gamma$-LQG universality class is always $-\cos(\pi\gamma^2/4)$. Mating-of-trees bijections for various random planar maps are studied in~\cite{mullin-maps,bernardi-maps,shef-burger,kmsw-bipolar,gkmw-burger,lsw-schnyder-wood,bernardi-dfs-bijection,bhs-site-perc}.

The mated-CRT map $\mcl G^\ep$ has a natural embedding into $\BB C$ which comes from the theory of SLE-decorated Liouville quantum gravity. 
Here we describe only the basic idea of this embedding. More details can be found in Section~\ref{sec-peanosphere} and a thorough treatment is given in the introductory sections of~\cite{ghs-dist-exponent}.  
Although ordinary SLE$_\kappa$ is space filling if and only if $\kappa\geq 8$, it was shown in~\cite{ig4} that a natural space-filling variant of SLE$_\kappa$ exists whenever $\kappa > 4$. For $\kappa \in (4,8)$, this variant recursively explores the bubbles that are cut off by an ordinary SLE$_\kappa$ as they are created.
Let $\eta$ be such a space-filling variant of SLE$_\kappa$ for $\kappa = 16/\gamma^2 >4$ which travels from $\infty$ to $\infty$ in $\BB C$, and suppose we parametrize $\eta$ by $\gamma$-LQG mass with respect to a certain independent $\gamma$-LQG surface called a \emph{$\gamma$-quantum cone}, which describes the local behavior of a GFF viewed from a point sampled from the $\gamma$-LQG measure. Then it follows from~\cite[Theorem 1.9]{wedges} that the mated-CRT map $\mcl G^\ep$ has the same law 
as the adjacency graph of ``cells" $\eta([x-\ep,x])$ for $x\in \ep \BB Z$, with two such cells considered to be adjacent if they intersect along a non-trivial connected boundary arc.
Thus we can embed $\mcl G^\ep$ into $\BB C$ via the map $x\mapsto \eta(x)$, which sends each vertex to the corresponding cell (see Figure~\ref{fig-mated-crt-map}, right panel).
 
\subsection{Main result in the general case}
\label{sec-main-results}

In this section we state our results in full generality. We begin by listing the random planar map models that our results apply to.
Each of the following is an infinite-volume random rooted planar maps $ (M, \BB v )  $, each equipped with its natural root vertex. In each case, the corresponding $\gamma$-LQG universality class is indicated in parentheses.\footnote{The main theorems of~\cite{ghs-map-dist,gm-spec-dim,dg-lqg-dim} also apply to one additional random planar map not listed here: the \emph{uniform infinite Schnyder wood-decorated triangulation}, as constructed in~\cite{lsw-schnyder-wood} ($\gamma = 1$). 
We expect that our results are also valid for this random planar map, but we exclude it to avoid dealing with certain technicalities (see Remark~\ref{remark-schnyder-wood}).}
\begin{enumerate}
\item The \emph{uniform infinite planar triangulation} (UIPT) of type II, which is the local limit of uniform triangulations with no self-loops, but multiple edges allowed~\cite{angel-schramm-uipt} ($\gamma=\sqrt{8/3}$). 
\item The \emph{uniform infinite spanning-tree decorated planar map}, which is the local limit of random spanning-tree weighted planar maps~\cite{shef-burger,chen-fk} ($\gamma = \sqrt 2$).
\item The \emph{uniform infinite bipolar oriented planar map}, as constructed in~\cite{kmsw-bipolar}\footnote{See~\cite[Section~3.3]{ghs-map-dist} for a careful proof that the infinite-volume bipolar-oriented planar maps considered in this paper exist as Benjamini-Schramm~\cite{benjamini-schramm-topology} limits of finite bipolar-oriented maps.} ($\gamma = \sqrt{4/3}$). 
\item More generally, one of the other distributions on infinite bipolar-oriented maps considered in~\cite[Section~2.3]{kmsw-bipolar} for which the face degree distribution has an exponential tail and the correlation between the coordinates of the encoding walk is $-\cos(\pi\gamma^2/4)$ (e.g., an infinite bipolar-oriented $k$-angulation for $k\geq 3$ --- in which case $\gamma=\sqrt{4/3}$ --- or one of the bipolar-oriented maps with biased face degree distributions considered in~\cite[Remark~1]{kmsw-bipolar} (see also~\cite[Section 3.3.4]{ghs-map-dist}), for which $\gamma \in (0,\sqrt 2)$).
\item The $\gamma$-mated-CRT map for $\gamma \in (0,2)$ with cell size $\ep=1$, as defined in Section~\ref{sec-mated-crt-map}.
\end{enumerate}
In the first four cases, $M$ comes with a distinguished root edge $\BB e$ and we let $\BB v$ be one of the endpoints of $\BB e$, chosen uniformly at random. In the case of the mated-CRT map the vertex set is identified with $\mathbb{Z}$ and we take $\BB v = 0$.   

\begin{defn} \label{def-walk} 
We write $X^M$ for the simple random walk on $M$ started from $\BB v$.
\end{defn}

The general version of our main result is an upper bound for the graph distance displacement of $X^M$. For the UIPT (and also the $\sqrt{8/3}$-mated-CRT map) we get an upper bound of $n^{1/4 + o_n(1)}$ for this displacement, which gives the correct exponent. For the other random planar maps listed at the beginning of this subsection, which belong to the $\gamma$-LQG universality class for $\gamma\not=\sqrt{8/3}$, we cannot explicitly compute the exponent for the graph distance displacement of the walk since we do not have exact expressions for the exponents which describe distances in the map. Computing such exponents is equivalent to computing the Hausdorff dimension of $\gamma$-LQG, which is one of the most important problems in the theory of LQG; see~\cite{ghs-dist-exponent,ding-goswami-watabiki,dzz-heat-kernel,dg-lqg-dim} for further discussion. 

However, we know from the results of~\cite{ghs-dist-exponent,ghs-map-dist,dzz-heat-kernel,dg-lqg-dim} that exponents for certain distances in these random planar maps exist. 
In particular, it is shown in~\cite[Theorem 1.6]{dg-lqg-dim} (building on results of~\cite{ghs-map-dist,dzz-heat-kernel}) that there exists for each $\gamma \in (0,2)$ an exponent $d_\gamma > 2$ which for any of the planar maps $(\mcl M  , \BB v)$ above is given by the a.s.\ limit
\eqb \label{eqn-dist-exponent} 
d_\gamma  = \lim_{r \rta\infty} \frac{ \log \# \mcl V\mcl B_r^{ M}(\BB v)   }{\log r} ,
\eqe 
where $   \mcl V\mcl B_r^{ M}(\BB v)$ denotes the vertex set of the graph-distance ball of radius $r$ centered at $\BB v$. 
Note that $d_{\sqrt{8/3}} = 4$ by~\cite[Theorem 1.2]{angel-peeling}. 
The reason for the notation $d_\gamma$ is that this exponent is expected to be the Hausdorff dimension of $\gamma$-LQG. 
The paper~\cite{dg-lqg-dim} also proves bounds for $d_\gamma$, shows that it is a continuous, strictly increasing function of $\gamma$, and (together with~\cite{dzz-heat-kernel}) shows that it describes several quantities associated with continuum LQG --- defined in terms of the Liouville heat kernel, Liouville graph distance, and Liouville first passage percolation. 
Our bounds for graph distances in random planar maps for general $\gamma \in (0,2)$ will be given in terms of $d_\gamma$.

\begin{thm} \label{thm-walk-speed} 
Let $(M,\BB v)$ be one of the random planar maps listed at the beginning of this section and let $\gamma \in (0,2)$ be the corresponding LQG parameter. Let $d_\gamma$ be as in~\eqref{eqn-dist-exponent}.
For each $\zeta \in (0,1)$, there exists $\alpha > 0$ (depending on $\zeta$ and the particular model) such that for each $n\in\BB N$, the simple random walk on $M$ satisfies 
\eqb \label{eqn-walk-speed}
\BB P\left[ \max_{1\leq j \leq n} 
 \op{dist}^M(X_j^M , \BB v) \leq n^{1/d_\gamma + \zeta}  \right] \geq 1 - O_n(n^{-\alpha}) .
\eqe 
Furthermore, a.s.\ 
\begin{align}
 \label{eqn-walk-speed-a.s.}
 \lim_{n\rta\infty} \frac{\log \max_{1\leq j \leq n } \op{dist}^M(  X_j^M , \BB v) }{\log n} &= \frac{1}{d_\gamma} .
 \end{align} 
\end{thm}

Theorem~\ref{thm-walk-speed-uipt} is the special case of Theorem~\ref{thm-walk-speed} when $(M,\BB v)$ is the UIPT. As noted after the statement of Theorem~\ref{thm-walk-speed-uipt}, the a.s.\ convergence~\eqref{eqn-walk-speed-a.s.} will follow from~\eqref{eqn-walk-speed} together with the corresponding lower bound in~\cite{gm-spec-dim}.

Let us now remark on the implications of Theorem~\ref{thm-walk-speed-uipt} in the case $\gamma\not=\sqrt{8/3}$. It is shown in~\cite[Theorem 1.2]{dg-lqg-dim} that the ball growth exponent $d_\gamma$ satisfies the bounds $\ul d_\gamma \leq d_\gamma \leq \ol d_\gamma $ for  
\eqb \label{eqn-d-lower}
\ul d_\gamma := 
\begin{dcases}
\max\left\{ \sqrt 6 \gamma ,  \frac{2\gamma^2}{4+\gamma^2-\sqrt{16 +\gamma^4}}         \right\} ,\quad &\gamma \leq \sqrt{8/3}  \\
\frac13 \left( 4 + \gamma^2 +\sqrt{16 + 2 \gamma^2 + \gamma^4} \right) ,\quad &\gamma \geq \sqrt{8/3} 
\end{dcases}
\eqe 
and
\eqb  \label{eqn-d-upper}
\ol d_\gamma := 
\begin{dcases}
\min\left\{    \frac13 \left( 4 + \gamma^2 +\sqrt{16 + 2 \gamma^2 + \gamma^4} \right)  ,  2 + \frac{\gamma^2}{2} + \sqrt 2 \gamma \right\} ,\quad &\gamma \leq \sqrt{8/3}  \\
\sqrt 6 \gamma ,\quad &\gamma \geq \sqrt{8/3} 
\end{dcases} .
\eqe 
See Figure~\ref{fig-d-bound} for a graph of the reciporicals of these upper and lower bounds (which correspond to our bounds for the walk speed exponent). 
Since $d_\gamma > 2$ for every $\gamma \in (0,2)$, Theorem~\ref{thm-walk-speed} shows that the random walk on each of the random planar maps considered in this paper is subdiffusive, with reasonably tight bounds for the subdiffusivity exponent. For example, in the case of the spanning-tree weighted map we have 
\eqb 
0.275255 \approx \frac{3}{6 + 2 \sqrt 6} \leq  \frac{1}{d_{\sqrt 2}}    \leq    \frac{1}{2\sqrt 3} \approx 0.288675 .
\eqe 
Further discussion of the source of the upper and lower bounds for $d_\gamma$ and their relationships to various physics predictions can be found in~\cite[Section 1.3]{dg-lqg-dim}.

\begin{figure}[t!]
 \begin{center}
\includegraphics[scale=.6]{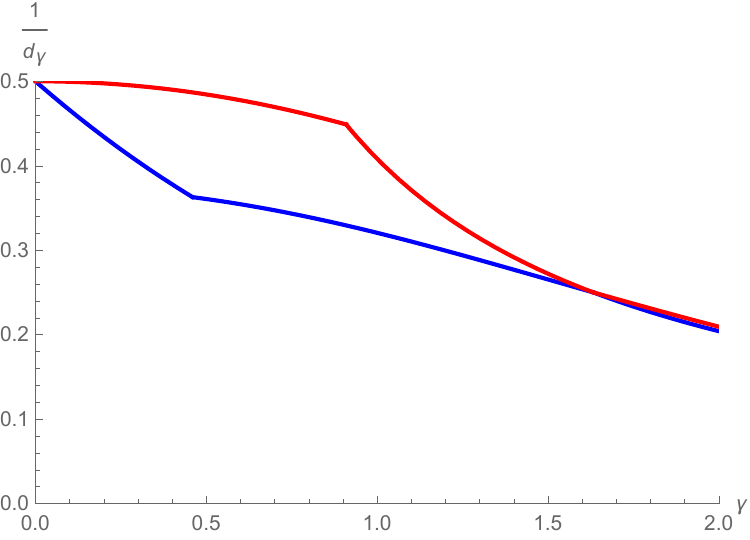} \hspace{15pt} \includegraphics[scale=.6]{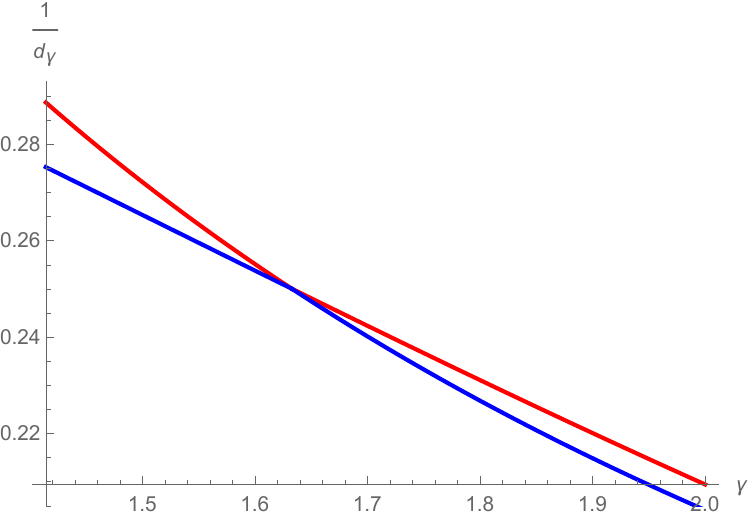}
\vspace{-0.01\textheight}
\caption{ \textbf{Left.} Graph of our upper and lower bounds for the subdiffusivity exponent $1/d_\gamma$ for $\gamma \in (0,2)$. Note that the bounds match only for $\gamma=\sqrt{8/3}$ (which corresponds to the UIPT case). 
\textbf{Right.} Graph of the same functions but restricted to the interval $[\sqrt 2 ,2]$. 
}\label{fig-d-bound}
\end{center}
\vspace{-1em}
\end{figure}

In the course of proving Theorem~\ref{thm-walk-speed}, we will obtain the exponent for the \emph{Euclidean} displacement of random walk on the mated-CRT map under its a priori (SLE/LQG) embedding, which is alluded to in Section~\ref{sec-mated-crt-map} and described in more detail in Section~\ref{sec-peanosphere}.

\begin{thm}[Euclidean displacement exponent] \label{thm-euc-displacement}
Let $\gamma \in (0,2)$ and let $X^{\mcl G^1}$ be the simple random walk on the mated-CRT map $\mcl G^1$ started from 0. Also let $\eta$ be the associated space-filling SLE curve parametrized by $\gamma$-LQG mass as in Section~\ref{sec-mated-crt-map}, so that $\BB Z\ni x\mapsto \eta(x) \in \BB C$ is the embedding of $\mcl G^1$ discussed in that section. Almost surely,
\eqb \label{eqn-euc-displacement}
\lim_{n\rta\infty} \frac{\log \max_{1\leq j \leq n} |\eta(X_j^{\mathcal G^1})|}{\log n} = \frac{1}{2-\gamma^2/2} .
\eqe
\end{thm}

The proof of Theorem~\ref{thm-euc-displacement} is explained at the end of Section~\ref{sec-euclidean-dist}. 
The upper bound is essentially an intermediate step in the proof of Theorem~\ref{thm-walk-speed}. The lower bound is a straightforward consequence of~\cite[Proposition 3.4]{gm-spec-dim}, which gives a logarithmic upper bound for the effective resistance to the boundary of a Euclidean ball. 

We note that Theorem~\ref{thm-euc-displacement} is consistent with the Euclidean displacement exponent for Liouville Brownian motion, which was computed by Jackson~\cite[Remark 1.5]{jackson-lbm-thick-pts}. We expect that the exponent $1/(2-\gamma^2/2)$ is universal across unimodular parabolic random planar maps in the $\gamma$-LQG universality class that are embedded in the plane in a conformally natural way. So, for example, the Euclidean distance traveled by an $n$-step random walk on the circle packing of the UIPT 
should be of order $n^{3/2}$, whereas on a spanning tree-weighted map this distance should be of order $n$.
 We do not investigate this further here, however.


\subsection{Basic notation}  
\label{sec-basic-notation} 
 
\noindent
\textbf{Integers.}
We write $\BB N$ for the set of positive integers and $\BB N_0 = \BB N\cup \{0\}$.  
For $a,b \in \BB R$ with $a<b$ and $r > 0$, we define the discrete intervals $[a,b]_{r \BB Z} := [a, b]\cap (r \BB Z)$ and $(a,b)_{r \BB Z} := (a,b)\cap (r\BB Z)$. 
\medskip

\noindent
\textbf{Asymptotics.}
If $a$ and $b$ are two quantities we write $a\preceq b$ (resp.\ $a \succeq b$) if there is a constant $C > 0$ (independent of the values of $a$ or $b$ and certain other parameters of interest) such that $a \leq C b$ (resp.\ $a \geq C b$). We write $a \asymp b$ if $a\preceq b$ and $a \succeq b$. 

If $a$ and $b$ are two quantities depending on a variable $x$, we write $a = O_x(b)$ (resp.\ $a = o_x(b)$) if $a/b$ remains bounded (resp.\ tends to 0) as $x\to 0$ or as $x\to\infty$ (the regime we are considering will be clear from the context). 
We write $a = o_x^\infty(b)$ if $a = o_x(b^s)$ for every $s\in\BB R$. 

We typically describe dependence of implicit constants and $O(\cdot)$ or $o(\cdot)$ errors in the statements of theorems, lemmas, and propositions, and require constants and errors in the proof to satisfy the same dependencies.
\medskip

 \noindent
\textbf{Euclidean space.}
For $K\subset \BB C$, we write $\op{Area}(K)$ for the Lebesgue measure of $K$ and $\op{diam}(K)$ for its Euclidean diameter.
For $r > 0$ and $z\in\BB C$ we write $B_r(z)$ for the open disk of radius $r$ centered at $z$. 
 
\noindent
\textbf{Graphs.} 
For a graph $G$, we write $\mcl V(G)$ and $\mcl E(G)$, respectively, for the set of vertices and edges of $G$, respectively. We sometimes omit the parentheses and write $\mcl VG = \mcl V(G)$ and $\mcl EG = \mcl E(G)$. For $v\in\mcl V(G)$, we write $\op{deg}^G(v)$ for the degree of $v$ (i.e., the number of edges with $v$ as an endpoint). 
For $r \geq 0$ and a vertex $v$ of $G$, we write $\mcl B_r^{G}(v)$ for the \emph{metric ball}, i.e., the subgraph of $G$ induced by the set of vertices of $G$ which lie at graph distance at most $r$ from $v$.

\subsection{Perspective and approach}
\label{sec-embedding-intro}

The first four random planar maps $M$ listed at the beginning of Section~\ref{sec-main-results} are special since these maps (when equipped with an appropriate statistical mechanics model) can be encoded by means of a mating-of-trees bijection for which the encoding walk $\mcl Z$ has i.i.d.\ increments. 
This allows us to couple $M$ with the mated-CRT map $\mcl G^\ep$ by coupling $\mcl Z$ with the two-dimensional Brownian motion $Z$ used to construct $\mcl G^\ep$ in~\eqref{eqn-inf-adjacency}. In particular, we couple $\mcl Z$ and $Z$ using the strong coupling theorem of Zaitsev~\cite{zaitsev-kmt} (which is a generalization of the KMT coupling~\cite{kmt} for walks which do not necessarily have nearest neighbor steps). It is shown in~\cite{ghs-map-dist} that under this coupling it holds with high probability that the following is true (see Section~\ref{sec-coupling} for details): Let $I\subset \BB R$ be a large interval. Then the (almost) submaps $M_{I}$ and $\mcl G^\ep_{\ep I}$ of $M$ and $\mcl G^\ep$, corresponding, respectively, to the time intervals $I \cap \BB Z$ and $\ep (I \cap \BB Z)$ for the encoding processes, are roughly isometric up to a polylogarithmic factor. That is, there is a function from $M_I$ to $\mcl G^\ep_{\ep I}$ which distorts graph distances by a factor of at most $O((\log |I|)^p )$ for a universal constant $p>0$. 

The above coupling is used in~\cite{ghs-map-dist,gm-spec-dim,dg-lqg-dim} to deduce estimates for the map $M$ from estimates for the mated-CRT map, which can in turn be proven using SLE/LQG theory due to the embedding $x\mapsto \eta(x)$ of the mated-CRT map discussed at the end of Section~\ref{sec-mated-crt-map}. 
 
In this paper, we will take a different perspective from the one in~\cite{ghs-map-dist,gm-spec-dim,dg-lqg-dim} in comparing $M$ to the mated-CRT map. Namely, we will first couple $M$ with the mated-CRT map $\mcl G^\ep$ as above with the length of the interval $I$ taken to be a large negative power of $\ep$, so that large subgraphs of $M$ and $\mcl G^\ep$ differ by a rough isometry. We will then study the embedding of (a large subgraph of) $M $ into $\BB C$ which is the composition of the rough isometry $M \rta \mcl G^\ep$ arising from our coupling and the embedding $x\mapsto \eta(x)$ of $\mcl G^\ep$. See Figure~\ref{fig-embedding-setup}.

\begin{figure}[t!]
 \begin{center}
\includegraphics[scale=.65]{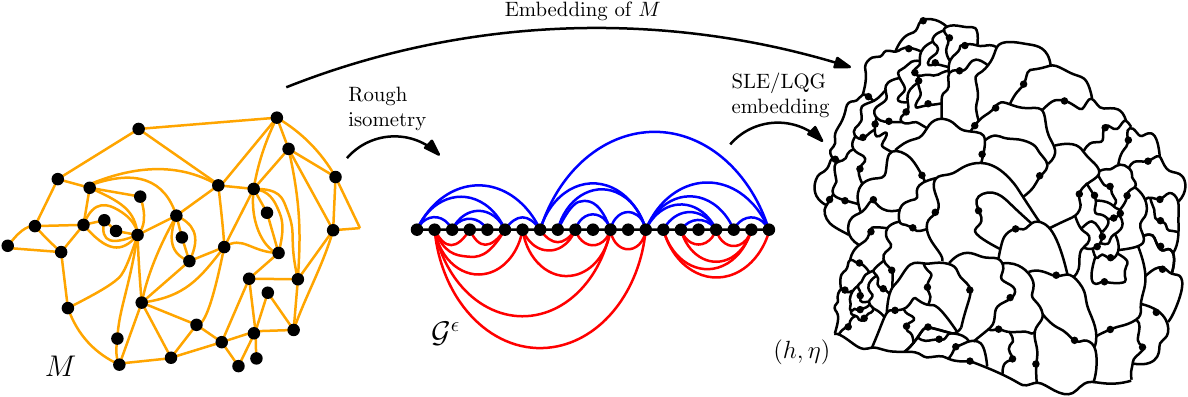}
\vspace{-0.01\textheight}
\caption{We study the embedding of (a large subgraph of) $M $ into $\BB C$ obtained by composing the rough isometry from this subgraph to a subgraph of $ \mcl G^\ep$ which comes from the coupling of~\cite{ghs-map-dist} with the embedding of $\mcl G^\ep$ into $\BB C$ which comes from the fact that $\mcl G^\ep$ is the adjacency graph of space-filling SLE cells with unit quantum mass.
}\label{fig-embedding-setup}
\end{center}
\end{figure} 

A number of papers have studied random planar maps by analyzing their embedding into $\BB C$ via the \emph{circle packing} (see~\cite{stephenson-circle-packing}  for an introduction). This is done in, e.g.,~\cite{benjamini-schramm-topology,gn-recurrence,abgn-bdy,gill-rohde-type,angel-hyperbolic,lee-conformal-growth,lee-uniformizing}.
Some of the techniques used in this paper are similar to ones used to analyze circle packings of random planar maps, but here our planar map is embedded into $\BB C$ using the embedding of Figure~\ref{fig-embedding-setup} rather than the circle packing embedding.  
 
Our embedding has several nice properties.
The space-filling SLE cells (and hence the faces of the embedding) are ``roughly spherical" in the sense that the ratio of their squared Euclidean diameter to their Lebesgue measure is unlikely to be large~\cite[Section 3]{ghm-kpz}. Moreover, the embedding we use also has several properties that are expected but not proven to hold for the circle packing. 
  For instance, the maximal diameter of the cells which intersect a Euclidean ball of fixed radius decays polynomially as $\ep\rta 0$ (Lemma~\ref{lem-max-cell-diam}). This means that the maximal Euclidean length of the embedded edges of $M$ which intersect a fixed Euclidean ball also decays polynomially as $\ep\rta 0$. Establishing the analogous statement for the circle packing of the UIPT is an open problem. We note that some recent progress on this problem has been made in~\cite{gjn-macroscopic-circles} which gives a criterion for a circle packing of a map to have macroscopic circles and proves that this criterion holds for the mated-CRT map.

For our purposes, one of the most important features of our the embedding is that the graph distance diameter of the set of vertices contained in a fixed Euclidean ball (with respect to either $\mcl G^\ep$ or $M$) under the above embedding is with high probabilty at most $\ep^{-1/d_\gamma + o_\ep(1)}$, with $d_\gamma$ as in~\eqref{eqn-dist-exponent}. This was proven in~\cite[Proposition 4.6]{dg-lqg-dim}. Again, the circle packings of the maps we consider are expected but not proven to have this property.

Our embedding gives rise to a weighting on the vertices of $M $ by assigning each vertex a weight equal to, roughly speaking, the diameter of the corresponding space-filling SLE cell (for various technical reasons we use a weight which is not exactly equal to this diameter). 
 This means that the \emph{weighted} graph distance between two embedded vertices approximates their Euclidean distance.
The weighting we consider is defined precisely in Section~\ref{sec-weight-def}. 

We will prove an upper bound for the displacement of the random walk on $M$ with respect to the weighted graph distance, and thereby the embedded Euclidean distance, using Markov-type theory, in particular the results of~\cite{dlp-markov-type}, as mentioned earlier in the introduction. 
We stress again that while Markov-type theory is typically used to prove diffusive upper bounds on the walk, our application is more subtle than this, since, in order to get useful bounds, we need to match up the scaling of the cell size $\ep$ with the number of steps taken by the walk. This is related to the fact that we get an exponent of $1/(2-\gamma^2/2)$ instead of $1/2$ in Theorem~\ref{thm-euc-displacement}. 

Due to the aforementioned comparison between graph distance balls and Euclidean balls, upon taking $\ep^{-1 + o_\ep(1)} = n$ the above upper bound for the Euclidean displacement of the walk gives us our desired upper bound for graph distance displacement and thereby concludes the proof of Theorem~\ref{thm-walk-speed}. We note that our basic strategy is similar to the proof of~\cite[Theorem 1.9]{lee-conformal-growth}, but we have a sharper comparison between weighted and unweighted graph distances than one has for the weighting used in~\cite{lee-conformal-growth}, so we get an optimal bound for the walk displacement exponent.

\subsection{Outline}
\label{sec-outline}

The rest of this paper is organized as follows. In Section~\ref{sec-prelim}, we review some definitions for random planar maps and weight functions on their vertices which originally appeared~\cite{aldous-lyons-unimodular,lee-conformal-growth}, record an extension of a Markov type inequality from~\cite{dlp-markov-type}, review some facts about SLE and LQG, and state the strong coupling result for various combinatorial random planar maps with the mated-CRT map which was proven in~\cite{ghs-map-dist}. 
Section~\ref{sec-core-argument} contains the main body of our proofs, following the approach discussed in Section~\ref{sec-embedding-intro}. Section~\ref{sec-estimates} contains the proofs of some technical estimates which are needed in Section~\ref{sec-core-argument}, but are deferred until later to avoid interrupting the main argument.

\section{Preliminaries}
\label{sec-prelim}

\subsection{Unimodular and reversible weighted graphs}
\label{sec-unimodular}

In this subsection we briefly review the definitions of unimodular and reversible random rooted graphs. We refer the reader to \cite{aldous-lyons-unimodular,angel-hyperbolic}, and the lecture notes \cite{CurienNotes} for a detailed development and overview of this theory.

A \emph{vertex-weighted graph} is a pair $(G,\omega)$ consisting of a graph $G$ and a \emph{weighting} on $G$, i.e., a function $\omega : \mcl V(G) \rta [0,\infty)$. A vertex-weighted graph possesses a natural weighted graph distance. A \emph{path} in $G$ is a function $P : [0,n]_{\BB Z} \rta \mcl V(G)$ for some $n\in\BB N$ such that $P(i)$ and $P(i-1)$ are either equal or connected by an edge in $G$ for each $i\in [1,n]_{\BB Z}$. We write $|P| = n$ for the \emph{length} of $P$. Given a weighted graph $G$ and vertices $v,w\in \mcl V(G)$, we define the \emph{weighted graph distance} by
\eqb \label{eqn-weighted-dist}
\op{dist}_\omega^G(v,w) := \inf_P \sum_{i=1}^{|P|} \frac12 \left( \omega(P(i)) + \omega(P(i-1)) \right) 
\eqe
where the infimum is over all finite paths $P$ in $G$ from $v$ to $w$ in $G$. 
 
Let $\BB G_{\bullet }^{\op{wt}}$ be the space of 3-tuples $(G,\omega , \BB v  )$ consisting of a connected locally finite graph $G$, a weighting on $G$, and a marked vertex of $G$. We equip $\BB G_\bullet^{\op{wt}}$ with the following obvious generalization of the Benjamini-Schramm local topology~\cite{benjamini-schramm-topology}: the distance from $(G,\omega,\BB v) $ to $ (G',\omega',\BB v')$ is the quantity $1/(N+1)$, where $N$ is the largest integer for which there exists a graph isomorphism $\psi : \mcl B_N^G(\BB v) \rta \mcl B_N^{G'}(\BB v')$ such that $|\omega'(\psi(v)) - \omega(v)| \leq 1/N$ for each $v \in \mcl V\mcl B_N^G(\BB v)$. 

We will be interested in \emph{unimodular} and \emph{reversible} random vertex-weighted graphs. For the definitions, we need to consider the space $\BB G_{\bullet \bullet}^{\op{wt}}$ consisting of vertex-weighted graphs with two marked vertices instead of one. The $\sigma$-algebra on $\BB G_{\bullet\bullet}^{\op{wt}}$ is the Borel $\sigma$-algebra corresponding to the topology induced by the metric defined as follows. Suppose that $(G,\omega,\BB v , u)$ and $(G',\omega',\BB v',u')$ are elements of $G_{\bullet\bullet}^{\op{wt}}$. Let $N_0$ (resp.\ $N_0'$) be the graph distance from $\BB v$ to $u$ in $G$ (resp.\ $\BB v'$ to $u'$ in $G'$). 
If either $N_0 \not= N_0'$ or $ \mcl B_{N_0}^G(\BB v)$ and $\mcl B_{N_0}^{G'}(\BB v')$ are not isomorphic as graphs, we define the distance between $(G,\omega,\BB v , u)$ and $(G',\omega',\BB v',u')$ to be 1. Otherwise, we define the distance to be $1/(N+1)$ where $N \geq N_0$ is the largest integer for which there exists a graph isomorphism $\psi : \mcl B_N^G(\BB v) \rta \mcl B_N^{G'}(\BB v')$ such that $\psi(u) = u'$ and $|\omega'(\psi(v)) - \omega(v)| \leq 1/N$ for each $v \in \mcl V\mcl B_N^G(\BB v)$.

\begin{defn}[Unimodular vertex-weighted graph] \label{def-unimodular-weight}
If $(G, \omega , \BB v)$ is a random element of $\BB G_{\bullet}^{\op{wt}}$, we say that $(G,\omega, \BB v)$ is a \emph{unimodular vertex-weighted graph} and $\omega$ is a \emph{unimodular vertex weighting} on $G$ if it satisfies the so-called \emph{mass transport principle}: for each Borel measurable function $F : \BB G_{\bullet\bullet}^{\op{wt}} \rta [0,\infty)$, 
\eqb \label{eqn-unimodular-sum}
\BB E\left[ \sum_{u \in \mcl V(G) } F(G , \omega , \BB v , u ) \right] 
 = \BB E\left[ \sum_{u \in \mcl V(G)} F(G , \omega  ,  u , \BB v) \right]  .
\eqe 
\end{defn}

Unweighted unimodular random rooted graphs are defined similarly.
A unimodular vertex weighting is called a \emph{conformal metric} in~\cite{lee-uniformizing,lee-conformal-growth}. We use the term ``unimodular vertex weighting" instead since we find it more descriptive.

\begin{defn}[Reversible vertex-weighted graph] \label{def-reversible-weight}
If $(G, \omega , \BB v)$ is a random element of $\BB G_{\bullet}^{\op{wt}}$, we say that $(G,\omega, \BB v)$ is a \emph{reversible vertex-weighted graph} and $\omega$ is a \emph{reversible vertex weighting} on $G$ if the following is true. Let $\wt{\BB v}$ be sampled uniformly from the set of neighbors of $\BB v$ in $M$. Then $(G,\omega,\BB v , \wt{\BB v}) \eqD (G,\omega,\wt{\BB v} , \BB v)$. 
\end{defn}

Note that if $(G,\omega,\BB v)$ is unimodular and satisfies $\mathbb{E} \deg \BB v < \infty$, then the random rooted vertex-weighted graph obtained by biasing the law of $(G,\omega,\BB v)$ by $\deg \BB v$ is reversible. Similarly, if $(G,\omega,\BB v)$ is reversible then the random rooted vertex-weighted graph obtained by biasing the law of $(G,\omega,\BB v)$ by $\deg^{-1} \BB v$ is unimodular. See \cite[Proposition 2.5]{benjamini-curien-stationary}.

\subsection{Markov-type inequalities}
\label{sec-markov-type}

In this section we review the notion of \textbf{Markov-type inequalities}, which will play a crucial role in our analysis. 

A metric space $\mathbb{X}=(\mathbb{X},d)$ is said to have \textbf{Markov-type $p$} if there exists a constant $C<\infty$ such that the following holds:  For every finite set $S$, every transition matrix $P$ of an irreducible reversible Markov chain on $S$, and every function $\phi:S\to \mathbb{X}$, we have that
\[\mathbb{E}\Bigl[d\bigl(\phi(X_0),\phi(X_n)\bigr)^p\Bigr] \leq C^p n  \mathbb{E}\Bigl[ d\bigl(\phi(X_0),\phi(X_1)\bigr)^p\Bigr]\]
for every $n\geq 0$, 
where $(X_0)_{n\geq0}$ is a sample of the Markov chain defined by $P$ with $X_0$ distributed according to the stationary measure of $P$. If $\mathbb{X}$ has Markov-type $p$, we refer to the optimal choice of $C$ as $M_p(\mathbb{X})$. Similarly, we say that $\mathbb{X}$ has \textbf{maximal Markov-type $p$} if there exists a constant $C<\infty$ such that
\[\mathbb{E}\left[\max_{0\leq m \leq n} d\bigl(\phi(X_0),\phi(X_m)\bigr)^p\right] \leq C^p n  \mathbb{E}\Bigl[ d\bigl(\phi(X_0),\phi(X_1)\bigr)^p\Bigr]\]
whenever $S,P,\phi$ and $X$ are as above and $n\geq 0$, and refer to the optimal choice of $C$ as $M_p^*(\mathbb{X})$.
We will be interested in applying these inequalities in the case that $p=2$, $S=\mathbb{X}= \mcl V(G)$ is the vertex set of a finite graph, $\phi$ is the identity function, and $X$ is the simple random walk on $G$.

Markov-type inequalities were first introduced by Ball~\cite{ball-markov-type}, who proved that Hilbert space has Markov-type 2. A powerful and elegant method for proving Markov-type inequalities was subsequently developed by Naor, Peres, Schramm, and Sheffield~\cite{npss-markov-type}, who proved Markov-type inequalities for many further examples including trees, hyperbolic groups, and $L^p$ for $p\geq 2$. Furthermore, in each space that they proved has Markov-type $2$, their proof also yielded that the space has \emph{maximal} Markov-type $2$~\cite[Section 8, Remark 8]{npss-markov-type} (it is an open problem to determine whether the two notions are equivalent).
Building upon this work, Ding, Lee, and Peres~\cite{dlp-markov-type} proved the following remarkable theorem.

\begin{thm}[Ding, Lee, and Peres]
\label{thm:DLP}
There exists a universal constant $C$ such that every vertex-weighted planar graph has Markov-type $2$ with $M_2\leq C$.
\end{thm}

In fact, the following maximal version of the Ding-Lee-Peres Theorem also follows implictly from their proof.

\begin{prop}
\label{prop:maximalDLP}
There exists a universal constant $C$ such that every vertex-weighted planar graph has maximal Markov-type $2$ with $M^*_2\leq C$.
\end{prop}

\begin{proof}[Proof of Proposition \ref{prop:maximalDLP}]
We give only a brief indication of the straightforward modifications to the proof of~\cite{dlp-markov-type} in order to deduce Proposition \ref{prop:maximalDLP} rather than Theorem~\ref{thm:DLP}. 
The proof of~\cite[Lemma 2.3]{dlp-markov-type} in fact establishes the maximal version of that lemma, in which the $\sup_{\xi \in I}\mathbb{P}(\|M^\xi_n-M^\xi_0\|\geq y)$ appearing in the integrand is replaced by $\sup_{\xi \in I}\mathbb{P}(\max_{1\leq t \leq n}\|M^\xi_t-M^\xi_0\|\geq y)$. Indeed, this stronger inequality appears as the final displayed inequality of the proof. (Note that there is a typo in this inequality, namely a factor of $y^{p-1}$ is missing from the integrand.) 
Once this maximal version of Lemma 2.3 is established, it is a simple matter to go through the proof of~\cite[Theorem 3.1]{dlp-markov-type}, adding maxima where appropriate and replacing the application of the original Lemma 2.3 with the maximal version.
\end{proof}

Rather than applying Theorem \ref{thm:DLP} and Proposition \ref{prop:maximalDLP} directly, although doing so is certainly possible, we will instead use them to deduce the following diffusivity estimate for random walks on (possibly) \emph{infinite}, \emph{hyperfinite}, unimodular random rooted planar graphs. We recall that a \textbf{percolation} on a unimodular random rooted graph $(G, \BB v)$ is a random labelling $\eta$ of the edge set of $G$ by elements of $\{0,1\}$ such that the resulting edge-labelled graph $(G,\eta,\BB v)$ is unimodular. 
We think of the percolation $\eta$ as a random subgraph of $G$, where an edge is labeled $1$ if it is included in the subgraph and $0$ otherwise, and denote the connected component of $\BB v$ by $K_\eta(\BB v)$. We say that a percolation is \textbf{finitary} if $K_\eta(\BB v)$ is almost surely finite, and say that a unimodular random rooted graph  $(G,\BB v)$ is \textbf{hyperfinite} if there exists an increasing sequence of finitary percolations $(\eta_n)_{n\geq1}$ on $(G ,\BB v)$ such that $\bigcup_{n\geq 1} K_{\eta_n} (\BB v) = \mcl V(G)$ almost surely. All these definitions extend naturally to vertex-weighted unimodular random rooted graphs, see \cite[Section 3.3]{angel-hyperbolic} for more detail. 

It is a corollary of the Lipton-Tarjan planar separator theorem \cite{LT80} that Benjamini-Schramm limits of finite planar graphs are always hyperfinite, see  \cite[Theorem 3.6]{angel-hyperbolic} and the proof of \cite[Corollary 4.5]{angel-hyperbolic}. As a consequence, all the graphs we consider in this paper are hyperfinite. (In fact, a unimodular random planar map is hyperfinite if and only if it is a Benjamini-Schramm limit of finite planar maps \cite{angel-hyperbolic}.) See \cite{angel-hyperbolic} for further background on these and related notions. 

\begin{cor}
\label{cor:UnimodularMarkovType}
Let $(G,\BB v)$ be a hyperfinite, unimodular random rooted graph with $\mathbb{E}[\deg(\BB v)]<\infty$ that is almost surely planar, and let $\omega$ be a unimodular vertex weighting of $G$. Then 
\begin{equation}
\label{eq:unimodularDLP}
\mathbb{E}\left[ \deg(\BB v) \max_{1\leq m \leq n} \op{dist}_\omega^G(\BB v,X_m^G)^2\right]\leq C^2 n \mathbb{E}\left[\deg(\BB v) \omega(\BB v)^2\right]
\end{equation}
for every $n\geq 0$, where $C$ is the universal constant from Proposition \ref{prop:maximalDLP}.
\end{cor}

It is an immediate consequence of Corollary \ref{cor:UnimodularMarkovType} that if $(G,\omega,\BB v)$ is an invariantly amenable, reversible, vertex-weighted random rooted graph that is almost surely planar then 
\begin{equation}
\label{eq:reversibleDLP}
\mathbb{E}\left[ \max_{1\leq m \leq n} \op{dist}_\omega^G(\BB v,X_m^G)^2\right]\leq C^2 n \mathbb{E}\left[ \omega(\BB v)^2\right].
\end{equation}
Indeed, this follows by applying Corollary~\ref{cor:UnimodularMarkovType} to the $\deg^{-1}(\BB v)$-biased version of $(G,\omega,\BB v)$, which is unimodular and hyperfinite.

\begin{proof}  
It suffices to consider the case that $\omega$ is almost surely bounded by some constant; the general case follows by truncating and applying the monotone convergence theorem. By scaling, we may assume without loss of generality that all the weights are in $[0,1]$ almost surely. 

Since $(G,\BB v)$ is hyperfinite, $(G,\omega, \BB v)$ is also. That is, there exists an increasing sequence of finitary percolations $(\eta_N)_{N\geq 1}$ on $(G,\omega,\BB v)$ such that $\bigcup_{N\geq1} K_{\eta_N}(\BB v) = V$ almost surely \cite[Lemma 3.2]{angel-hyperbolic}. Let $(G_N,\omega_N,\BB v)$ be the subgraph of $G$ induced by $K_{\eta_N}(\BB v)$, together with the restriction of $\omega_N$ to $K_{\eta_N}(\BB v)$. It follows by a well-known application of the mass-transport principle \cite[Lemma 3.1]{angel-hyperbolic} that conditional on the isomorphism class of $(G_N,\omega_N)$, the root $\BB v$ is uniformly distributed on the vertex set of $G_N$. Thus, if we bias the law of $(G_N,\omega_N,\BB v)$ by the degree of $\BB v$ in $G_N$, then, conditional on the isomorphism class of $(G_N,\omega_N)$, $\BB v$ is distributed according to the stationary measure of the random walk on $G_N$. 
Applying Proposition \ref{prop:maximalDLP} we obtain that
\[
\mathbb{E}\left[ \deg^{G_N}(\BB v) \max_{1 \leq m \leq n}\op{dist}_{\omega_N}^{G_N}(\BB v,X_m^{G_N})^2\right]\leq C^2 n \mathbb{E}\left[\deg^{G_N}(\BB v) \omega_N(\BB v)^2\right]
\]
for every $n,N\geq 1$.
Since the two random variables we are taking the expectations of are bounded by the integrable random variables $n^2 \deg(\BB v)$ and $\deg(\BB v)$ respectively, we can take $N\to\infty$ and apply the dominated convergence theorem to deduce the claimed inequality.
\end{proof}
 
\subsection{Liouville quantum gravity}  
\label{sec-lqg-prelim}

The Gaussian free field (GFF) is the canonical random distribution (generalized function) on a domain $D\subset \BB C$. We assume that the reader is familiar with the GFF and refer to~\cite{shef-gff,ss-contour,ig1,ig4,pw-gff-notes} for background. 

For $\gamma \in (0,2)$, a \emph{$\gamma$-Liouville quantum gravity (LQG)} surface is a random surface described by some variant $h$ of the GFF on a domain $D\subset \BB C$ whose Riemannian metric tensor is given formally by $e^{\gamma h(z)} \,dx \otimes dy$, where $dx \otimes dy$ is the Euclidean metric tensor. This definition does not make rigorous sense since the GFF is a distribution, not a function.

However, one can, to an extent, make rigorous sense of $\gamma$-LQG surfaces via various regularization procedures. It was shown in~\cite{shef-kpz} that one can define the \emph{$\gamma$-LQG area measure} $\mu_h$ associated with a $\gamma$-LQG surface by the formula
\eqb \label{eqn-mu_h-def}
\mu_h = \lim_{\ep\rta 0} e^{\gamma h_\ep(z)} \,dz
\eqe 
where $dz$ is Lebesgue measure, $h_\ep(z)$ is the circle average of $h$ over the circle $\bdy B_\ep(z)$ (see~\cite[Section 3.1]{shef-kpz} for the construction and basic properties of circle averages), and the limit takes place a.s.\ with respect to the Prokhorov topology as $\ep\rta 0$ along powers of 2. 
A similar regularization procedure yields the \emph{$\gamma$-LQG boundary length measure} $\nu_h$ which is defined on certain curves including $\bdy D$ and SLE$_\kappa$-type curves for $\kappa=16/\gamma^2$ that are independent from $h$~\cite{shef-zipper,benoist-lqg-chaos}. There is also a more general theory of regularized measures of this type, called \emph{Gaussian multiplicative chaos}, which was initiated by Kahane~\cite{kahane} and is surveyed in~\cite{rhodes-vargas-review,berestycki-gmt-elementary}. 

The measures $\mu_h$ and $\nu_h$ satisfy a conformal covariance formula~\cite[Proposition 2.1]{shef-kpz}: if $f: \wt D \rta D$ is a conformal map, $h$ is some variant of the GFF on $D$ (such as an embedding of the $\gamma$-quantum cone, defined below) and 
\eqb \label{eqn-lqg-coord}
\wt h = h \circ f + Q\log |f'| \quad \op{for} \quad Q = \frac{2}{\gamma} + \frac{\gamma}{2} 
\eqe
then $f_* \mu_{\wt h} = \mu_h$ and $f_* \nu_{\wt h} = \nu_h$. 

We think of two pairs $(D,h)$ and $(\wt D ,\wt h)$ which are related as in~\eqref{eqn-lqg-coord} as two different parameterizations of the same $\gamma$-LQG surface. 
This leads us to define a $\gamma$-LQG surface to be an equivalence class of pairs $(D,h)$ consisting of a domain $D$ and a distribution $h$ on $D$ (which we will always take to be random, and indeed to be some variant of the GFF) with two such pairs $(D,h)$ and $(\wt D , \wt h)$ declared to be equivalent if they are related by a conformal map as in~\eqref{eqn-lqg-coord}. More generally, we can define a \emph{$\gamma$-LQG surface with $k \in \BB N$ marked points} to be an equivalence class of $k+2$-tuples $(D,h,z_1,\dots,z_k)$ where $D$ is an open subset of $\BB C$, $h$ is a distribution on $D$, and $z_1,\dots,z_k \in D\cup \bdy D$, with two such $k+2$-tuples declared to be equivalent if they differ by a conformal map $f$ as in~\eqref{eqn-lqg-coord} which takes the marked points for one surface to the corresponding marked points for the other surface.

We call a particular choice of equivalence class representation $(D,h,z_1,\dots,z_k)$ an \emph{embedding} of the surface into $(D,z_1,\dots,z_k)$. 

\subsubsection{The $\gamma$-quantum cone}
\label{sec-cone-prelim}

The main type of $\gamma$-LQG surface that we will be interested in in this paper is the \emph{$\gamma$-quantum cone}, which was first defined in~\cite[Definition~4.10]{wedges}.  The $\gamma$-quantum cone is a doubly marked $\gamma$-LQG surface which can be represented by $(\BB C ,h , 0,\infty)$, where the distribution $h$ is a slightly modified version of $\wt h -\gamma \log |\cdot|$, where $\wt h$ is a whole-plane GFF. Roughly speaking, the $\gamma$-quantum cone describes the local behavior of a GFF near a point sampled from the $\gamma$-LQG measure (this follows from~\cite[Lemma A.10]{wedges}, which says that the GFF has a $\gamma$-log singularity near such a point, and~\cite[Proposition 4.13(ii)]{wedges}). The precise definition of the embedding $h$ will be important for our purposes, so we give it here.  

Let $A : \BB R \to \BB R$ be the process $A_t:= B_t + \gamma t$, where $B_t$ is a standard linear Brownian motion conditioned so that $ B_t  - (Q-\gamma) t > 0$ for all $t< 0$ 
(see~\cite[Remark 4.4]{wedges} for an explanation of how to make sense of this singular conditioning as a Doob transform). In particular, $(B_t)_{t\geq 0}$ is an unconditioned standard linear Brownian motion. Let $h$ be the random distribution such that if $h_r(0)$ denotes the circle average of $h$ on $\partial B_r(0)$ (as in~\eqref{eqn-mu_h-def}), then $t\mapsto h_{e^{-t}}(0)$ has the same law as the process $A$; and $h - h_{|\cdot|}(0)$ is independent from $h_{|\cdot|}(0)$ and has the same law as the analogous process for a whole-plane GFF.\footnote{Here, $h-h_{|\cdot|}(0)$ is the distribution obtained by subtracting the continuous function $z\mapsto h_{|z|}(0)$ from the distribution $h$. This distribution has the property that its average over every circle centered at the origin is 0.}

The above definition only gives us one possible embedding of the $\gamma$-quantum cone, which we call the \emph{circle average embedding}. One obtains an equivalent $\gamma$-LQG surface by replacing $h$ with $h(a\cdot) + Q\log |a|$ for any $a\in \BB C$, with $Q$ as in~\eqref{eqn-lqg-coord}. The circle average embedding is characterized by the properties that $\sup_{r>0}\{ h_r(0) + Q\log r =0\} =1$ (which follows since $A_0 = 0$) and that $h|_{\BB D}$ agrees in law with  
$(h'-\gamma\log|\cdot| - h'_1(0))|_{\BB D}$, where $h'$ is a whole-plane GFF and $h'_1(0)$ is the circle average of $h'$ over $\partial \mathbb{D}$. 
  
The $\gamma$-quantum cone possesses a scale invariance property which is different from the scale invariance of the law of the whole-plane GFF. To state this property, define
\eqb \label{eqn-mass-hit-time}
R_b  = R_b(h) := \sup\left\{ r > 0 : h_r(0) + Q \log r = \frac{1}{\gamma} \log b \right\} ,\quad \forall b > 0 .
\eqe 
That is, $R_b$ gives the largest radius $r > 0$ so that if we scale spatially by the factor $r$ and apply the change of coordinates formula~\eqref{eqn-lqg-coord}, then the average of the resulting field on $\partial \BB D$ is equal to $\gamma^{-1} \log b$.  
Note that $R_0 = 1$ by the definition of the circle average embedding. It is easy to see from the definition of $h$ (and is shown in~\cite[Proposition~4.13(i)]{wedges}) that 
\eqb \label{eqn-cone-scale}
h \eqD h(R_b \cdot) + Q \log R_b -  \frac{1}{\gamma} \log b ,\quad \forall b > 0 .
\eqe
Furthermore, for $b_2>b_1>0$, $-\log(R_{b_2}/R_{b_1})$ has the same law as the first time that a standard linear Brownian motion with negative linear drift $-(Q-\gamma) t$ hits $\frac{1}{\gamma} \log(b_2/b_1)$. 

By~\eqref{eqn-lqg-coord}, if we let $h^b$ be the field on the right side of~\eqref{eqn-cone-scale}, then a.s.\ $\mu_{h^b}(A) = b \mu_h(R_b^{-1} A)$ for each Borel set $A\subset \BB C$. In particular, $\mu_h(B_{R_b})$ is typically of order $b$. We will frequently use the following elementary estimate for $R_b$ (see~\cite[Lemma 2.1]{gms-tutte} for a proof).

\begin{lem}[\!\!\cite{gms-tutte}] \label{lem-cone-hit-tail}
There is a constant $a = a( \gamma)  > 0$ such that for each $b_2>b_1>0$ and each $C > 1$,
\begin{align} \label{eqn-cone-hit-tail}
 \BB P\left[ C^{-1} (b_2/b_1)^{ \tfrac{1 }{\gamma (Q-\gamma) }} \leq R_{b_2}/R_{b_1} \leq C (b_2/b_1)^{ \tfrac{1 }{\gamma (Q-\gamma)}} \right] 
  \geq 1  -   3 \exp\left( - \frac{  a  (\log C)^2 }{   \log (b_2/b_1) +  \log C   } \right)  .
\end{align}
\end{lem}

\subsection{The SLE/LQG description of the mated-CRT map}
\label{sec-peanosphere}

In this subsection we describe in more detail the relationship between mated-CRT maps and SLE-decorated Liouville quantum gravity, as alluded to in Section~\ref{sec-mated-crt-map}. Let $\gamma \in (0,2)$ and let $\kappa := 16/\gamma^2 > 4$.

\emph{Whole-plane space-filling SLE$_\kappa$} from $\infty$ to $\infty$ is a variant of SLE$_\kappa$ that fills space, even in the case $\kappa \in (4,8)$, and a.s.\ hits Lebesgue-a.e.\ point of $\BB C$ exactly once. 
This variant of SLE was first introduced in~\cite[Sections 1.2.3 and 4.3]{ig4} (see also~\cite[Section~1.4.1]{wedges}).  

For $\kappa \geq 8$, whole-plane space-filling SLE$_\kappa$ is just a two-sided variant of ordinary SLE$_\kappa$. It can be obtained from chordal SLE$_\kappa$ by ``zooming in" near a Lebesgue-typical point $z $ at positive distance from the boundary of the domain.
For $\kappa \in (4,8)$, chordal space-filling SLE$_\kappa$ is obtained from ordinary SLE$_\kappa$ by iteratively filling in the bubbles disconnected from the target point by ordinary SLE$_\kappa$-type curves to get a space-filling curve (which is not a Loewner evolution). One can then obtain whole-plane space-filling SLE$_\kappa$ by zooming in near a point at positive distance from the boundary, as in the case $\kappa\geq 8$. We will not need the precise definition of whole-plane space-filling SLE$_\kappa$ here. 

Suppose now that $\eta$ is a whole-plane space-filling SLE$_\kappa$ and $(\BB C ,h , 0 ,\infty)$ is a $\gamma$-quantum cone (Section~\ref{sec-cone-prelim}) independent from $\eta$. Let $\mu_h$ and $\nu_h$ be the associated $\gamma$-LQG area and boundary length measures. We assume that $\eta$ is parametrized in such a way that $\eta(0) = 0$ and $\mu_h(\eta([t_1,t_2])) = t_2-t_1$ for each $t_1,t_2\in\BB R$ with $t_1<t_2$. We define the \emph{left boundary length process} $(L_t)_{t\in\BB R}$ as follows. We set $L_0 = 0$ and for $t_1 < t_2$ we require that $L_{t_2}-L_{t_1}$ gives the $\nu_h$-length of the intersection of the left outer boundaries of $\eta([t_1,t_2])$ and $\eta([t_2,\infty))$, minus the $\nu_h$-length of the intersection of the left outer boundaries of $\eta((-\infty,t_1])$ and $\eta([t_1,t_2])$. We similarly define the right outer boundary length process $(R_t)_{t\in\BB R}$ with ``right" in place of ``left". We set $Z_t := (L_t,R_t)$. 

It is shown in~\cite[Theorem 1.9]{wedges} that $Z$ evolves as a correlated two-sided two-dimensional Brownian motion with $\op{Corr}(L_t,R_t) = -\cos(\pi\gamma^2/4)$ and in~\cite[Theorem 1.11]{wedges} that $Z$ a.s.\ determines $h$ and $\eta$, modulo rotation and scaling. 

Let $\ep >0$. It is easy to see from the definition of $Z$ that two of the \emph{cells} $\eta([x_1-\ep,x_1])$ and $\eta([x_2-\ep,x_2])$ for $x_1,x_2\in\ep\BB Z$ with $x_1<x_2$ intersect along a non-trivial connected boundary arc if and only if the mated-CRT map adjacency condition~\eqref{eqn-inf-adjacency} holds for the above Brownian motion $Z$. Therefore, the mated-CRT map $\mcl G^\ep$ constructed from $Z$ is identical to the graph with vertex set $\ep\BB Z$, with two distinct vertices $x_1,x_2\in\ep\BB Z$ connected by an edge if and only if the corresponding space-filling SLE cells $\eta([x_1-\ep,x_1])$ and $\eta([x_2-\ep,x_2])$ intersect along a non-trivial connected boundary arc (the vertices are connected by two edges if $|x_1-x_2|  > \ep$ and the corresponding cells intersect along both their left and right boundary arcs). 

Thus, we obtain an embedding of $\mcl G^\ep$ into $\BB C$ by sending $x\in\ep\BB Z$ to the point $\eta(x)$. We remark that~\cite[Theorem 1.11]{wedges} does \emph{not} give an explicit description of this embedding in terms of $Z$. It is shown in~\cite{gms-tutte} that $x\mapsto \eta(x)$ is close to the so-called Tutte embedding of $\mcl G^\ep$ (which is defined by requiring the position of each vertex to be the average of the positions of its neighbors) when $\ep$ is small. We will not need this fact here, however. 

We introduce the following notation for the subgraph of $\mcl G^\ep$ corresponding to a domain $D\subset \BB C$. 

\begin{defn}  \label{def-sg-domain} 
For $\ep  > 0$ and a set $D\subset \BB C$, we write $\mcl G^\ep(D)$ for the subgraph of $\mcl G^\ep$ induced by the set of vertices 
 $x\in\ep\BB Z$ with $\eta([x-\ep,x])\cap D\not=\emptyset$. 
\end{defn}

An important property of the embedding $x\mapsto \eta(x)$ is that the maximal size of the cells that intersect a fixed Euclidean ball is of order $\ep^{2/(2+\gamma)^2  + o_\ep(1)}$ with high probability as $\ep\rta 0$, and in particular tends to 0 as $\ep\rta 0$. A quantitative version of this property is given by the following lemma, which follows from basic SLE/LQG estimates (see~\cite[Lemma~2.4]{gms-harmonic} for a proof).

\begin{lem}[\!\cite{gms-harmonic}] \label{lem-max-cell-diam} 
Suppose we are in the setting described just above.
For each $q \in \left( 0 , \tfrac{2}{(2+\gamma)^2} \right)$, each $r\in (0,1)$, and each $\ep \in (0,1)$, 
\eqb \label{eqn-max-cell-diam}
\BB P\left[ \op{diam}\left( \eta([x-\ep,x] \right) \leq \ep^q ,\: \forall x \in \mcl G^\ep( B_r(0) ) \right] \geq 1 -  \ep^{\alpha(q,\gamma) + o_\ep(1)}  ,
\eqe  
where the rate of the $o_\ep(1)$ depends only on $q$, $r$, and $\gamma$ and
\eqb \label{eqn-cell-diam-exponent}
\alpha(q,\gamma) := \frac{q}{2\gamma^2  } \left( \frac{1}{q} - 2 - \frac{\gamma^2}{2} \right)^2 - 2 q  .
\eqe 
\end{lem}

We note that the exponent $\alpha(q,\gamma)$ from~\eqref{eqn-cell-diam-exponent} tends to $\infty$ as $q \rta 0$ (in fact, this is the only property of this exponent that we will use).

\subsection{Strong coupling with the mated-CRT map}
\label{sec-coupling}

A key tool in this paper is a coupling result for any one of the first four random planar maps $(M,\BB e)$ listed in Section~\ref{sec-main-results} with the mated-CRT map $\mcl G^\ep$ for $\ep > 0$, which was introduced in~\cite{ghs-map-dist}. This coupling is based on a mating-of-trees bijection that encodes $(M,\BB e, \mcl S)$---for an appropriate statistical mechanics model $\mcl S$ on $M$---by means of a two-dimensional random walk $\mcl Z$ with i.i.d.\ increments (the particular step distribution depends on the model). The coupling of $M$ and $\mcl G^\ep$ is then obtained via a strong coupling of the random walk $\mcl Z$ with the Brownian motion used to define the mated-CRT map~\cite{zaitsev-kmt}. Let us now discuss the statistical mechanics model that we will consider on each of our random planar maps. 
\begin{enumerate}
\item In the case of the UIPT of type II, $\mcl S$ is critical ($p=1/2$) site percolation on $M$, or equivalently a uniform depth-first-search tree on $M$.  This bijection was introduced in~\cite{bernardi-dfs-bijection} in the setting of a uniform depth-first-search tree on a finite triangulation. The paper~\cite{bhs-site-perc} explains the connection to site percolation and the (straightforward) extension to the UIPT.
\item In the case of the infinite spanning-tree decorated planar map, $\mcl S$ is a uniform spanning tree on $M$~\cite{mullin-maps,bernardi-maps,shef-burger}.
\item For infinite bipolar-oriented planar maps of various types, $\mcl S$ is a uniformly chosen orientation on the edges of $M$ with no source or sink (i.e., the source and sink are equal to~$\infty$)~\cite{kmsw-bipolar}.
\end{enumerate}
These bijections are each reviewed in~\cite{ghs-map-dist}. We will not need the precise definitions of the bijections here. 
 
In each of the above cases, we let $\mcl G^\ep$ for $\ep > 0$ be the $\gamma$-mated-CRT map with cell size $\ep$, where $\gamma$ is the LQG parameter corresponding to $M$ as listed in Section~\ref{sec-main-results}. We assume that the maps $\{\mcl G^\ep\}_{\ep>0}$ are all constructed from a common correlated two-sided two-dimensional Brownian motion $Z=(L,R)$ with correlation $-\cos(\pi\gamma^2/4)$, as in Section~\ref{sec-mated-crt-map}. 

The coupling result of~\cite{ghs-map-dist} actually gives a coupling of certain large ``almost submaps" of $\mcl G^\ep$ and $M$, which we now discuss. 
  
\begin{defn} \label{def-sg-restrict}
For an interval $I\subset \BB R$ and $\ep \in (0,1)$, we write $\mcl G_{I}^\ep$ for the subgraph of $\mcl G^\ep$ induced by the vertex set $I\cap (\ep\BB Z)$. 
\end{defn}

The analogue of Definition~\ref{def-sg-restrict} for $(M,\BB e , \mcl S)$, as given in~\cite{ghs-map-dist}, is somewhat more complicated. 
For each of the combinatorial random planar maps listed above, the mating-of-trees bijection gives rise to a mapping $\lambda $ from $\BB Z$ to the edge set of $M$. One way to see how this mapping arises is as follows. If $i\in\BB Z$, then the translated walk $\mcl Z_{\cdot +i} - \mcl Z_i$ has the same law as $\BB Z$, so that applying the bijection to this translated walk produces a rooted, decorated planar map with the same law as  $(M,\BB e , \mcl S)$. The planar map (without the statistical mechanics model) is isomorphic to $M$, but with a different choice of root edge. This root edge is $\lambda(i)$. Note that $\lambda(0) = \BB e$. 

The mapping $i\mapsto \lambda(i)$ is a bijection in the case of the UIPT and bipolar-oriented maps, and is two-to-one in the case of the case of spanning tree-weighted maps. In the terminology of~\cite{shef-burger}, the two integers corresponding to a single edge are the indices of a ``burger" and of the ``order" that consumes it.
 
Recall that a \emph{planar map with boundary} is a planar map $M$ together with a distinguished face (called the \emph{external face}). The \emph{boundary} $\bdy M$ of $M$ is the subgraph of $M$ consisting of the vertices and edges on the boundary of the external face. We say that $M$ has \emph{simple boundary} if $\bdy M$ is a a simple cycle, i.e., it is isomorphic to the cycle $\mathbb{Z}/n\mathbb{Z}$ for some $n\geq 0$.

We can use the mating-of-trees bijection to define for each interval $I = [a,b] \subset \BB R$ a planar map~$M_{I}$ with boundary $\bdy M_{I}$ associated\footnote{In fact, in~\cite{ghs-map-dist} the definition of $M_{I}$ is only given in the case $I= [-n,n]$ for $n\in\BB N$, and $M_{I}$ is denoted by $M_n$. However, by translation invariance a completely analogous definition works for an arbitrary interval $I$.} 
with the random walk increment $(\mcl Z-\mcl Z_a)|_{ I \cap \BB Z}$, which is a discrete analogue of $\mcl G_I^\ep$ from Definition~\ref{def-sg-restrict}, and is almost but not exactly equal to the submap of $M$ spanned by the edge set $\lambda(I \cap \mathbb{Z})$.
Indeed, as explained in~\cite{ghs-map-dist}, due to the possibility of pairs of vertices or edges being identified at a time after the right endpoint of $I$, we cannot in general take $M_I$ to be a subgraph of $M$ if we want Theorem~\ref{thm-map-count} below to hold. The precise definition of $M_I$ is slightly different in each of the above cases, and is given in~\cite[Section 3]{ghs-map-dist}. For our purposes, the most important property of $M_I$ is that there is an ``almost inclusion" map 
\eqb \label{eqn-inclusion-function}
\iota_{I} : M_{I} \to M \quad \text{which is injective on $M_{I} \setminus \bdy M_{I}$} .
\eqe  
So, we can canonically identify $M_{I} \setminus \bdy M_{I}$ with a subgraph of $M$. By~\cite[Remark 1.3]{ghs-map-dist}, for each interval $I$,
\eqb \label{eqn-submap-edge} 
\lambda^{-1} \Bigl( \iota_I\bigl( \mcl E M_I  \setminus \mcl E ( \bdy M_I )   \bigr) \Bigr) \subset I\cap \BB Z 
\quad \op{and} \quad 
\lambda( I\cap \BB Z ) \subset    \iota_I \left( \mcl E M_I  \right) \cup \{\lambda(\lfloor b \rfloor)\}  . 
\eqe
If $0\in I$ and $\lambda(0) \in \mcl E M_I \setminus \mcl E(\bdy M_I)$, then $M_I$ possesses a canonical root edge which is mapped to $\BB e$ by $\iota_I$. By a slight abuse of notation, we will denote this root edge by $\BB e$. 

One can also define for each interval $I\subset \BB R$ functions 
\eqb \label{eqn-peano-functions}
\phi_{I} : \mcl V(M_{I}) \to I \cap\BB Z  \quad \op{and} \quad \psi_{I} : I\cap \BB Z \to \mcl V(M_{I}) .
\eqe     
Roughly speaking, the vertex $\psi_{I}(i)$ for $i\in I$ corresponds to the $i$th step of the walk $\mcl Z $ in the bijective construction of $(M,e_0,T)$ from $\mcl Z$ and $\phi_{I}$ is ``close" to being the inverse of $\psi_{I}$. 
However, the construction of $M $ from $\mcl Z$ does not set up an exact bijection between $I\cap \BB Z$ and the vertex set of $M_{I} $, so the functions $\phi_{I}$ and $\psi_{I}$ are neither injective nor surjective. As is the case for $M_I$, the definitions of $\phi_I$ and $\psi_I$ are slightly different in each case and are given in~\cite[Section 3]{ghs-map-dist}. 
See Figure~\ref{fig-coupling-setup} for an illustration of the above objects.

 \vspace{1em}
\begin{figure}[ht!]
 \begin{center}
\includegraphics[scale=.7]{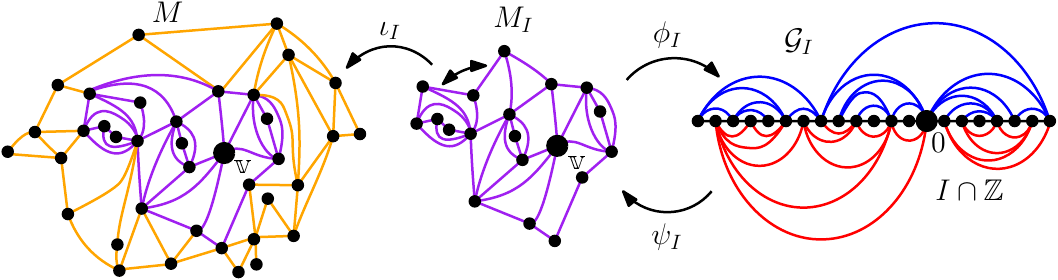} 
\caption{Illustration of the map $M_{I}$, the ``almost inclusion" map $\iota_{I} : M_{I}\rta M$, and the maps $\phi_{I} : \mcl V(M_{I}) \rta I\cap\BB Z$ and $\psi_{I} : I \cap {\BB Z} \rta  \mcl V(M_{I})  $. Note that two edges of $\bdy M_{I}$ get identified when we apply $\iota_{I}$ in the figure. Theorem~\ref{thm-map-count} tells us that the maps $\phi_I$ and $\psi_I$ are rough isometries up to polylogarithmic errors when $I\cap\BB Z$ is equipped with the graph structure coming from the mated-CRT map $\mcl G_I$. A similar figure appears in~\cite{ghs-map-dist}. 
}\label{fig-coupling-setup}
\end{center}
\vspace{-1em}
\end{figure}

The following is the coupling result that we will use in this paper. It is a trivial modification of~\cite[Theorem 1.9]{ghs-map-dist}.

\begin{thm}
\label{thm-map-count}
Let $(M,\BB e,\mathcal{S})$ be one of the random map models coupled with a statistical mechanics model that is listed at the beginning of Section \ref{sec-coupling} and let $\phi_{I}$ and $\psi_{I}$ be as above for each bounded interval $I \subset \BB R$.
There is a universal constant $p_0 > 4$ such that for each $\ep\in (0,1)$, each $n\in\BB N$, and each interval $I\subset \BB R$ with length $n$, there is a coupling of the Brownian motion $Z$ used to construct $\mcl G^\ep$ with $(M,\BB e , \mcl S)$ such that the following is true with probability $1- o_n^\infty(n)$ (at a rate depending only on the law of $(M,\BB e,\mcl S)$).
\begin{enumerate} 
\item For each edge $\{v_1,v_2\} \in  \mcl E(M_{I})$, there is a path from $ \ep \phi_{I}(v_1)$ to $ \ep \phi_{I}(v_2)$ in $\mcl G_{\ep I}^\ep$ with length at most $ (\log n)^{p_0}$.  \label{item-map-count-G}
\item For each edge $\{x_1,x_2\} \in \mcl E\mcl G^\ep_{\ep I}$, there is a path from $\psi_{I}(x_1/\ep )$ to $\psi_{I}(x_2 /\ep)$ in $M_{I}$ with length at most $(\log n)^{p_0}$.  \label{item-map-count-M} 
\item We have $\op{dist}^{M_{I}}\left(  \psi_{I}( \phi_{I}(v)) , v  \right) \leq  ( \log n)^{p_0}$ for each $v\in \mcl V(M_{I})$ and $ \op{dist}^{\mcl G_{I}}\left(\ep \phi_{I}(\psi_{I}(x/\ep)) , x  \right) \leq  (\log n)^{p_0} $ for each $x \in \ep ( I\cap\BB Z)$. \label{item-map-count-close}
\end{enumerate}
\end{thm}
 
Theorem~\ref{thm-map-count} follows from~\cite[Theorem 1.9]{ghs-map-dist} by using translation invariance to transfer from the interval $[-n,n] $ to the interval $I$; using Brownian scaling to transfer from $\mcl G^1$ to $\mcl G^\ep$; and choosing $p_0$ slightly larger than the exponent 4 appearing in that theorem in order to get rid of the constant $C$. Note that~\cite[Theorem 1.9]{ghs-map-dist} includes conditions on the number of paths that hit a given vertex which are not included in Theorem~\ref{thm-map-count} since they will not be needed in the present paper. 

We note that concatenating the paths from conditions~\ref{item-map-count-G} and~\ref{item-map-count-M} in Theorem~\ref{thm-map-count} shows that $\phi_I$ and $\psi_I$ distort graph distances by a factor of at most $(\log n)^{p_0}$, i.e., $\phi_I$ and $\psi_I$ are rough isometries with probability $1-o_n^\infty(n)$ (see~\cite[Lemma 1.10]{ghs-dist-exponent}).

\begin{remark}[Schnyder wood-decorated maps] \label{remark-schnyder-wood}
All of the results in this paper should also hold in the case of the uniform infinite Schnyder-wood decorated triangulation, as constructed in~\cite{lsw-schnyder-wood} (which belongs to the $\gamma$-LQG universality class for $\gamma = 1$, and which \emph{is} included in the main theorems of~\cite{ghs-map-dist,gm-spec-dim,dg-lqg-dim}).
The reason why we do not include this map in our list in Section~\ref{sec-main-results} is that the condition~\eqref{eqn-submap-edge} is not satisfied for the definition of $M_I$ used in the case of Schnyder wood-decorated maps in~\cite{ghs-map-dist} (see~\cite[Remark 1.3]{ghs-map-dist}).  
We expect that one can give an alternative definition of $M_I$ in the Schnyder wood-decorated case for which~\eqref{eqn-submap-edge} is satisfied using the bijection of~\cite{lsw-schnyder-wood}, and prove (via a few pages of combinatorial arguments) that all of the results in~\cite{ghs-map-dist} are still satisfied with this definition.
This would allow us to extend all of our results to the case of the Schnyder wood-decorated triangulation. 
\end{remark}

\section{The core argument}
\label{sec-core-argument}

In this section we will give the proof of our main results modulo a few technical estimates which are proven in Section~\ref{sec-estimates}. We follow the approach described in Section~\ref{sec-embedding-intro}. We start in Section~\ref{sec-core-setup} by defining a coupling of the combinatorial random planar map $M$ and the mated-CRT map $\mcl G^\ep$ using Theorem~\ref{thm-map-count} for a particular choice of interval $I=I^\ep$, thereby defining an embedding of $M_{I^\ep}$ into $\BB C$. In Section~\ref{sec-weight-def} we use our coupling to define a unimodular vertex weighting of $\mcl G^\ep$ and a reversible vertex weigting of $M_{I^\ep}$ (recall the definitions from Section~\ref{sec-unimodular}). We record several estimates for these weight functions in Section~\ref{sec-weight-estimate}. In Section~\ref{sec-euclidean-dist}, we prove an upper bound for the weighted graph distance displacement, and thereby the embedded Euclidean displacement, of the random walks on $M_{I^\ep}$ and $\mcl G^\ep$, using Markov type theory (in particular Corollary~\ref{cor:UnimodularMarkovType}). In Section~\ref{sec-graph-dist}, we conclude by comparing embedded Euclidean distances and graph distances, using the results of~\cite{dg-lqg-dim}.

\subsection{Setup}
\label{sec-core-setup}

Let $(M,\BB e)$ be one of the first four rooted planar maps listed in Section~\ref{sec-main-results} and let $\gamma \in (0,2)$ be the corresponding LQG parameter. We define the statistical mechanics model $\mcl S$ on $M$ and the maps $M_I$, the ``almost inclusion" $\iota_I : M_I \rta M$, and the functions $\phi_I : \mcl V  M_I \rta I\cap \BB Z$ and $\psi_I : I\cap\BB Z\rta \mcl V M_I$ for intervals $I\subset \BB R$ as in Section~\ref{sec-coupling}. (Let us note that all our analysis of the $\gamma$-mated CRT map  applies for arbitrary $\gamma \in (0,2)$. To improve readability, however, we will henceforth restrict attention to those $\gamma \in (0,2)$ for which there is a corresponding combinatorial map model.)

Let $Z = (L,R)$ be the correlated Brownian motion from~\eqref{eqn-bm-cov} and let $\{\mcl G^\ep\}_{\ep > 0}$ be the $\gamma$-mated-CRT maps with spacing $\ep$ constructed from $Z$ as in~\eqref{eqn-inf-adjacency}. 
Let $((\BB C ,h , 0, \infty) , \eta)$ be the $\gamma$-quantum cone/space-filling SLE$_{\kappa}$ curve pair determined by $Z$ as in Section~\ref{sec-peanosphere}.  We assume that $h$ is a circle average embedding (Section~\ref{sec-lqg-prelim}) and that $\eta$ is parametrized by $\gamma$-LQG mass with respect to $h$, so that $\mcl G^\ep$ is isomorphic to the adjacency graph of cells $\eta([x-\ep,x])$ for $x\in\ep\BB Z$.  

We will now couple the above objects together using Theorem~\ref{thm-map-count} with a particular choice of parameters. Fix a large constant $K > 1$, which we will eventually choose in a manner which depends only on $\gamma$. For $\ep \in (0,1)$, let $\theta^\ep$ be sampled uniformly from $[0, \ep^{-K} ]_{\BB Z}$, independently from everything else. We couple $Z$ and $(M,\BB e,\mcl S)$ as in Theorem~\ref{thm-map-count} with the coupling interval 
\eqb \label{eqn-coupling-interval}
I^\ep  = [a^\ep ,b^\ep] := [-\theta^\ep ,  \ep^{-K} - \theta^\ep ]  
\eqe
and with $n = \lfloor \ep^{-K} \rfloor$. We emphasize that the superscript $\ep$ does not denote a power (we will never be raising anything to the power $\ep$). The reason for the random index shift $\theta^\ep$ is to avoid making the root vertex $\BB v$ special from the perspective of the interval $I^\ep$. In fact, we have the following lemma, which will be important when we check that a certain vertex weighting on $\mcl VM_{I^\ep}$ is reversible (Lemma~\ref{lem-unimodular-weight}). 
 
\begin{lem} \label{lem-uniform-edge} 
The planar map $M_{I^\ep}$ (but not the interval $I^\ep$) is a.s.\ determined by the translated random walk $\left( \mcl Z_{t -\theta^\ep} - \mcl Z_{-\theta^\ep} \right)_{t \in \BB Z}$.
Moreover, if we condition on the random walk / Brownian motion pair 
\eqb \label{eqn-translated-pair}
 \left( \left( Z_{t- \ep \theta^\ep}  - Z_{-\ep \theta^\ep} \right)_{t\in\BB R} , \left( \mcl Z_{t -\theta^\ep} - \mcl Z_{-\theta^\ep} \right)_{t\in \BB Z} \right)  
\eqe 
and on the event $ \left\{\BB e \in \mcl E( M_{I^\ep} ) \setminus \mcl E(\bdy M_{I^\ep})  \right\}$, then the conditional law of the root edge $\BB e$ is uniform on $\mcl E ( M_{I^\ep}) \setminus \mcl E(\bdy M_{I^\ep} )$.
\end{lem}
\begin{proof}  
The definition of our coupling (see the discussion just after Theorem~\ref{thm-map-count}) implies that the translated pair~\eqref{eqn-translated-pair} 
is coupled together as in Theorem~\ref{thm-map-count} with coupling interval $ I  = [0, \lfloor \ep^{-K} \rfloor ]  $. The index shift $\theta^\ep$ is independent from the pair~\eqref{eqn-translated-pair} and the map $M_{I^\ep}$ is obtained from the translated walk in~\eqref{eqn-translated-pair} in the same manner that $M_{ [0, \lfloor \ep^{-K} \rfloor  ]}$ is obtained from $\mcl Z$. 

Let $\lambda : \BB Z\rta \mcl E(M)$ be the function defined just after Definition~\ref{def-sg-restrict} and let $\lambda_{\theta^\ep}(\cdot) := \lambda(\cdot - \theta^\ep) $ be the analogous function with the translated walk of~\eqref{eqn-translated-pair} in place of $\mcl Z$.  
Then the edge-rooted map $(M , \lambda(-\theta^\ep))$ and the path $\lambda_{\theta^\ep}$ are a.s.\ determined by the translated walk in~\eqref{eqn-translated-pair} in the same manner that $(M,\BB e)$ and $\lambda$ are a.s.\ determined by $\mcl Z$. 
By~\eqref{eqn-inclusion-function}, we can canonically identify $M_{I^\ep} \setminus \bdy M_{I^\ep}$ with the submap $\iota_{I^\ep}(M_{I^\ep} \setminus \bdy M_{I^\ep})$ of $M$.
We henceforth make this identification, so in particular the set $\lambda_{\theta^\ep}^{-1}\left(   \mcl E( M_{I^\ep} ) \setminus \mcl E(\bdy M_{I^\ep})   \right) \subset \BB Z$ is well-defined and determined by the walk in~\eqref{eqn-translated-pair}.
By~\eqref{eqn-submap-edge}, we have $\lambda_{\theta^\ep}^{-1}\left(  \mcl E( M_{I^\ep} ) \setminus \mcl E(\bdy M_{I^\ep}) \right) \subset [0, \lfloor \ep^{-K}\rfloor ]_{\BB Z}$. Furthermore, by the discussion just after Definition~\ref{def-sg-restrict}, $\lambda_{\theta^\ep}$ takes $\lambda_{\theta^\ep}^{-1}\left(    \mcl E( M_{I^\ep} ) \setminus \mcl E(\bdy M_{I^\ep})  \right)$ to $  \mcl E( M_{I^\ep} ) \setminus \mcl E(\bdy M_{I^\ep})  $ in either a one-to-one or a two-to-one manner (depending on the particular model under consideration).

We have $\lambda_{\theta^\ep}(\theta^\ep) = \lambda(0) = \BB e$. 
Since $\theta^\ep$ is sampled uniformly from $[0,\lfloor \ep^{-K} \rfloor  ]_{\BB Z}$ and $\theta^\ep$ is independent from the pair~\eqref{eqn-translated-pair}, if we condition on the pair~\eqref{eqn-translated-pair} and the event $ \left\{\BB e \in \mcl E( M_{I^\ep} ) \setminus \mcl E(\bdy M_{I^\ep})  \right\}$, then the conditional law of $\theta^\ep$ is uniform on the set $\lambda_{\theta^\ep}^{-1}\left(   \mcl E( M_{I^\ep} ) \setminus \mcl E(\bdy M_{I^\ep})   \right)$. 
Therefore, the last part of the preceding paragraph implies that the conditional law of $\lambda_{\theta^\ep}(\theta^\ep) = \BB e$ given~\eqref{eqn-translated-pair} and the event $\left\{\BB e \in \mcl E( M_{I^\ep} ) \setminus \mcl E(\bdy M_{I^\ep}) \right\}$ is uniform on $\mcl E ( M_{I^\ep}) \setminus \mcl E(\bdy M_{I^\ep} )$, as required.
\end{proof}

With $\phi_{I^\ep}$ and $\psi_{I^\ep}$ as in~\eqref{eqn-peano-functions}, we define
\eqb \label{eqn-coupling-function}
\phi^\ep := \ep \phi_{I^\ep} : \mcl V M_{I^\ep}  \rta \mcl V\mcl G^\ep_{I^\ep} , \quad
\psi^\ep :=   \psi_{I^\ep}(\cdot/\ep)  :  \mcl V\mcl G^\ep_{I^\ep} \rta \mcl V M_{I^\ep}   ,  
\quad \op{and} \quad
\Phi^\ep := \eta\circ \phi^\ep : \mcl VM_{I^\ep} \rta \BB C . 
\eqe 
The function $\Phi^\ep$ is the embedding of the map $M_{I^\ep}$ into $\BB C$ illustrated in Figure~\ref{fig-embedding-setup}. This embedding is obtained by first applying the function $\phi^\ep$ to send each vertex of $M_{I^\ep}$ to a corresponding vertex of $\mcl G^\ep_{I^\ep}$, then applying the embedding $\ep\BB Z\ni x\mapsto \eta(x)$ of $\mcl G^\ep_{I^\ep}$ into $\BB C$.

\subsection{Vertex weightings on $\mcl G^\ep$ and $M_{I^\ep}$}
\label{sec-weight-def}
 
We will use the embeddings $\eta: \mcl V\mcl G^\ep \rta \BB C$ and $\Phi^\ep : \mcl VM^\ep_{I^\ep} \rta \BB C$ to define a unimodular vertex weighting $\omega^\ep_{\mcl G}$ on $ \mcl G^\ep$ and a reversible vertex weighting $\omega^\ep_M$ on $ M^\ep_{I^\ep}$ (see Section~\ref{sec-unimodular} for definitions). We want to define these weightings in such a way that
\eqb \label{eqn-weight-approx}
\omega_{\mcl G}^\ep(x) \approx \op{diam}\Bigl(\eta\bigl([x-\ep,x]\bigr) \Bigr) \quad\op{and}\quad
\omega_M^\ep(v) \approx \op{diam}\Bigl( \eta\bigl([\phi^\ep(v) -\ep,\phi^\ep(v) ]\bigr) \Bigr) ,
\eqe 
so that $\omega_{\mcl G}^\ep$- and $\omega_M^\ep$-distances are close to Euclidean distances.  
However, rather than define $\omega_{\mcl G}^\ep$ and $\omega_M^\ep$ directly via the formulas in~\eqref{eqn-weight-approx}, we will use slightly modified weights which will be chosen to circumvent the following two inconveniences: 
\begin{enumerate}
\item The law of the pair $(h,\eta)$ is not exactly invariant under translations of the form $(h,\eta) \mapsto (h(\cdot +\eta(t)) , \eta(\cdot +t ) -\eta(t))$ for $t\in\BB R$, so~\eqref{eqn-weight-approx} does not define a reversible vertex weighting of $(M,\BB v)$ or a unimodular vertex weighting of $(\mathcal{G}^\ep,0)$. \label{item-issue-unimodular} 
\item The boundaries of the cells $\eta([\phi^\ep(v) -\ep,\phi^\ep(v) ])$ and $\eta([\phi^\ep(v') -\ep,\phi^\ep(v') ])$ corresponding to adjacent vertices $v,v' \in \mcl V M_{I^\ep}$ do not necessarily intersect and intersecting cells do not necessarily correspond to adjacent vertices, so the Euclidean distance between the embeddings of two vertices of $\mcl V M_{I^\ep}$ might not be comparable to the minimal sum of the diameters of the cells along a path in $M_{I^\ep}$ which connects the two given vertices. \label{item-issue-diam}
\end{enumerate} 
To get around issue~\ref{item-issue-unimodular}, we will define vertex weightings in terms of a re-scaled version of the pair $(h(\cdot +\eta(t)) , \eta(\cdot +t ) -\eta(t))$ whose law \emph{is} stationary in $t$. To get around issue~\ref{item-issue-diam}, we will define weights in terms of the Euclidean diameter of the union of the cells in a $\mcl G^\ep$-graph distance neighborhood of polylogarithmic size, and use that cells corresponding to adjacent vertices of $M_{I^\ep}$ cannot lie at more than polylogarithmic graph distance from each other in $\mcl G^\ep $ due to Theorem~\ref{thm-map-count}. We will show in Section~\ref{sec-weight-estimate} below that these modifications typically alter the weights by at most a subpolynomial factor as compared to~\eqref{eqn-weight-approx}. 

By~\cite[Theorem~1.9]{wedges}, and as previously discussed in Section \ref{sec-cone-prelim}, for each $t\in\BB R$ the field/curve pair $(h(\cdot +\eta(t)) , \eta(\cdot +t ) -\eta(t))$ agrees in law with $(h,\eta)$ modulo rotation and scaling, i.e., for each $t \in \BB R$ there is a random constant $\rho_t \in \BB C$ such that with
\eqb  \label{eqn-scale-field}
h^t :=  h(\rho_t\cdot + \eta(t)) + Q\log |\rho_t| 
\quad\text{and} \quad \eta^t := \rho_t^{-1} \left( \eta(\cdot  +t) - \eta(t) \right) 
\quad \text{we have} \quad
(h^t ,\eta^t) \eqD (h,\eta) .
\eqe 
The value of $|\rho_t|$ can be made explicit: since $h$ is a circle average embedding of a $\gamma$-quantum cone, we want to choose $\rho_t$ in such a way that $h^t$ is a circle average embedding. This means that $\op{arg}\rho_t$ is not necessarily determined by $h^t$ and by the definition of the circle average embedding (c.f.\ Section~\ref{sec-lqg-prelim}), 
\eqb \label{eqn-shift-def}
|\rho_t| = \sup\left\{ r > 0 : h_r(\eta(t)) + Q\log r  = 0 \right\} ,
\eqe 
where $h_r(\cdot)$ denotes the circle average process. We note that $|\rho_t|$ is not necessarily determined by $h^t$.

Now fix $p > p_0$, where $p_0$ is the constant from Theorem~\ref{thm-map-count}. For $x\in \ep \BB Z = \mcl V\mcl G^\ep$, define the weight
\allb \label{eqn-vertex-weight}
\omega_{\mcl G}^\ep(x) 
&:=  \max\left\{ 1  ,  \op{diam}\left( \bigcup_{y\in \mcl V\mcl B^{\mcl G^\ep}_{(\log \ep^{-1})^p}(0)} \eta^x([y-\ep,y]) \right) \right\}   \notag \\
&=   \max\left\{ 1  ,   |\rho_x|^{-1} \op{diam}\left( \bigcup_{y\in \mcl V\mcl B^{\mcl G^\ep}_{(\log \ep^{-1})^p}(x)} \eta ([y-\ep,y]) \right)  \right\}  .
\alle
Using the notation~\eqref{eqn-coupling-function}, we also define a weight on the vertices $v$ of $M_{I^\ep}$ by
\eqb \label{eqn-map-weight}
\omega_M^\ep (v) :=  
\omega_{\mcl G}^\ep( \phi^\ep(v)  )   . 
\eqe 

It turns out that $\omega_M^\ep$ is not quite a reversible vertex weighting on $M_{I^\ep}$, but it is a reversible vertex weighting on the connected component of the root edge in the (possibly disconnected) graph
\eqb \label{eqn-interior-graph}
\rng M_{I^\ep} := M_{I^\ep} \setminus \mcl E( \bdy M_{I^\ep} ) 
\eqe 
if we condition on the high-probability event that the root edge is not in $\mcl E(\bdy M_{I^\ep})$.

\begin{lem} \label{lem-unimodular-weight}
The vertex-weighted graph $(\mcl G^\ep , \omega_{\mcl G}^\ep , 0)$ is unimodular in the sense of Definition~\ref{def-unimodular-weight}. 
Furthermore, under the conditional law given that the root edge $\BB e$ belongs to $\mcl E( M_{I^\ep} ) \setminus \mcl E(\bdy M_{I^\ep})$, the vertex-weighted graph $\left( \rng M_{I^\ep} , \omega_M^\ep , \BB v \right)$ (using the notation~\eqref{eqn-interior-graph}) is reversible in the sense of Definition~\ref{def-reversible-weight}. 
\end{lem}
\begin{proof}
We first check unimodularity of $(\mcl G^\ep , \omega_{\mcl G}^\ep , 0)$. 
In the notation of~\eqref{eqn-scale-field}, for $x\in\mcl V\mcl G^\ep = \ep\BB Z$ the vertex-weighted graph $(\mcl G^\ep , \omega_{\mcl G}^\ep , x)$ is constructed from $(h^x,\eta^x)$ in the same deterministic manner that the vertex-weighted graph $(\mcl G^\ep , \omega_{\mcl G}^\ep , 0)$ is constructed from $(h,\eta)$. 
By this,~\eqref{eqn-scale-field}, and the invariance of the law of $Z$ under time reversal, it follows that for each such $x$, 
\eqb \label{eqn-vertex-reverse}
 (\mcl G^\ep , \omega_{\mcl G}^\ep ,  x, 0 ) \eqD  (\mcl G^\ep , \omega_{\mcl G}^\ep ,  0, - x )  \eqD (\mcl G^\ep , \omega_{\mcl G}^\ep ,  0, x ) 
\eqe
where here by $\eqD$ we mean equality in law as doubly marked vertex-weighted graphs (i.e., we forget the ordering of the vertices of $\mcl G^\ep$). 
Hence, for any choice of measurable function $F : \BB G_{\bullet\bullet}^{\op{wt}} \rta [0,\infty)$ as in Definition~\ref{def-unimodular-weight},
\alb
\BB E\left[ \sum_{x \in \mcl V\mcl G^\ep} F(\mcl G^\ep , \omega_{\mcl G}^\ep , 0,x ) \right] 
=  \sum_{x \in \ep\BB Z}  \BB E\left[ F(\mcl G^\ep , \omega_{\mcl G}^\ep , 0,x ) \right] 
&=  \sum_{x \in \ep\BB Z}  \BB E\left[ F(\mcl G^\ep , \omega_{\mcl G}^\ep , x , 0) \right] \\
&= \BB E\left[ \sum_{x \in \mcl V\mcl G^\ep} F(\mcl G^\ep , \omega_{\mcl G}^\ep , x, 0 ) \right]  .
\ale

Next we consider $\left( \rng M_{I^\ep} , \omega^\ep_M , \BB v \right)$. By the definitions~\eqref{eqn-vertex-weight} and~\eqref{eqn-map-weight} of the weight functions, the fact that $Z$ a.s.\ determines the pair $(h,\eta)$ modulo rotation~\cite[Theorem 1.11]{wedges}, and the same argument given in the proof of Lemma~\ref{lem-uniform-edge}, the translated Brownian motion/random walk pair of~\eqref{eqn-translated-pair} a.s.\ determines $(M_{I^\ep}  , \omega_M^\ep , \BB v)$. Consequently, Lemma~\ref{lem-uniform-edge} implies that the root edge $\BB e$ is uniformly distributed on $\mcl E(M_{I^\ep}) \setminus \mcl E( \bdy M_{I^\ep}) =\mcl E(\rng M_{I^\ep})$ if we condition on $(M_{I^\ep} ,\omega_M^\ep)$ and on the event $\left\{\BB e \in \mcl E( M_{I^\ep} ) \setminus \mcl E(\bdy M_{I^\ep}) \right\}$. Therefore, under this conditioning, $\BB v$ is sampled from the uniform measure on vertices of $\rng M_{I^\ep}$ weighted by their $\rng M_{I^\ep}$-degree, and so $( \rng M_{I^\ep} , \omega_M^\ep , \BB v)$ is reversible. 
\end{proof}

\subsection{Estimates for the weight functions} 
\label{sec-weight-estimate}

In this subsection we prove estimates for the vertex weight functions $\omega_{\mcl G}^\ep$ and $\omega_M^\ep$ defined in the preceding subsection which will eventually allow us to prove that (a) the random walks on $\mcl G^\ep$ and $M_{I^\ep}$ are diffusive with respect to the weighted graph distance and (b) under the embeddings $x\mapsto \eta(x)$ and $v\mapsto \Phi^\ep(v)$ discussed in Section~\ref{sec-core-setup}, $\omega_{\mcl G}^\ep$- (resp.\ $\omega_M^\ep$-) distances are close, modulo subpolynomial errors, to Euclidean distances.  
Some of the more standard arguments needed in this subsection are given in Section~\ref{sec-estimates} so as to allow the reader to get to the main ideas of the proof as soon as possible.

We first record an estimate for the second moment of the weight functions at the root vertex. This estimate is needed when we apply Corollary~\ref{cor:UnimodularMarkovType} in the proof of Lemma~\ref{lem-weight-dist} (this corollary is also the reason for the factor of $\op{deg}^{\mcl G^\ep}(0)$). 

\begin{prop} \label{prop-weight-moment}
Define the weight functions $\omega_{\mcl G}^\ep$ and $\omega_M^\ep$ as in~\eqref{eqn-vertex-weight} and~\eqref{eqn-map-weight}, respectively. For each $\ep \in (0,1)$,  
\eqb \label{eqn-weight-moment}
\BB E\left[ \omega_{\mcl G}^\ep(0)^2 \op{deg}^{\mcl G^\ep}(0) \right] \leq \ep^{1+o_\ep(1)} \quad \op{and} \quad
\BB E\left[ \omega_{M}^\ep(0)^2 \right] \leq \ep^{1+o_\ep(1)} .
\eqe  
\end{prop}
   
It is easy to see heuristically why we get a bound of order $\ep^{1+o_\ep(1)}$ in Proposition~\ref{prop-weight-moment}. The root cell $\eta([-\ep,0])$ looks roughly like a uniform (with respect to the counting measure on cells) cell of $\mcl G^\ep$ which intersects $\BB D$. Since there are typically of order $\ep^{-1}$ such cells, the expected Lebesgue measure of $\eta([-\ep,0])$ should be of order $\ep$. This Lebesgue measure is very unlikely to be much smaller than the squared Euclidean diameter of $\eta([-\ep,0])$ due to the estimates in~\cite[Section 3.2]{ghm-kpz}. Replacing a single cell with the union of the cells in a ball of polylogarithmic size should increase the diameter by at most a polylogarithmic factor. The degree factor in~\eqref{eqn-weight-moment} should not have a significant effect on the expectation since $\op{deg}^{\mcl G^\ep}(0)$ has an exponential tail~\cite[Lemma 2.2]{gms-harmonic}.
Rigorously, Proposition~\ref{prop-weight-moment} is a straightforward application of some basic SLE/LQG estimates, as will be explained in Section~\ref{sec-log-ball-moment}. 

The second main estimate of this subsection will allow us to deal with the distortion factor $|\rho_x|^{-1}$ appearing in~\eqref{eqn-vertex-weight} and thereby compare weighted graph distances to Euclidean distances.

\begin{prop} \label{prop-max-shift}
Define the scaling factors $\rho_t$ for $t\in\BB R$ as in~\eqref{eqn-scale-field}.
There exists $\alpha = \alpha( \gamma)  >0$ such that for each $S >1$, 
\eqb \label{eqn-max-shift}
\BB P\left[ \sup_{t \in \eta^{-1}(\BB D) } |\rho_t| \leq S \right] \geq 1 - O_S(S^{-\alpha})  
\eqe
and for each $\ep \in (0,1)$, 
\eqb \label{eqn-weight-compare}
\BB P\left[  \op{diam}\left( \bigcup_{y\in \mcl V\mcl B^{\mcl G^\ep}_{(\log \ep^{-1})^p}} \eta([x-\ep,x]) \right)  \leq S \omega_{\mcl G}^\ep(x)  ,\: \forall x \in \mcl V\mcl G^\ep(B_{1/2}(0) ) \right] \geq 1 - O_S(S^{-\alpha})  - o_\ep^\infty(\ep) .
\eqe
\end{prop}

Proposition~\ref{prop-max-shift} will be a straightforward consequence of the following lemma.

\begin{lem} \label{lem-cone-max}
Let $h$ be as in Section~\ref{sec-core-setup}, so that $h$ is a circle average embedding of a $\gamma$-quantum cone and let $h_r(z)$ for $r>0$ and $z\in\BB C$ be the circle average of $h$ over $\bdy B_r(z)$. 
For each $\zeta \in (0,1)$, there exists $\alpha=\alpha(\zeta, \gamma)>0$ such that for each $\delta \in (0,1)$, 
\eqb  \label{eqn-cone-max}
\BB P\left[ |h_r(z) + \gamma \log r| \leq  \zeta \log r ,\:  \forall r \geq S ,\: \forall z \in B_{r/2}(0) \right] \geq 1- O_S(S^{-\alpha}) .
\eqe 
\end{lem}

Lemma~\ref{lem-cone-max} is proven in Section~\ref{sec-cone-max} using elementary estimates for the circle average process of a GFF, which come from the fact that this process is Gaussian, with explicit covariance structure~\cite[Section 3.1]{shef-kpz}. 

\begin{proof}[Proof of Proposition~\ref{prop-max-shift}]
Recalling the formula~\eqref{eqn-shift-def} for $|\rho_t|$, we see that if the maximum in~\eqref{eqn-max-shift} is larger than $S$, then there is a $z\in\BB D$ and an $r > S$ such that $h_r(z) + Q\log r = 0$. This implies that $|h_r(z)  + \gamma \log r| \geq (Q-\gamma) \log r$ and that $z \in B_{r/2}(0)$ provided $S>2$. The estimate~\eqref{eqn-max-shift} for an appropriate choice of $\alpha$ therefore follows from Lemma~\ref{lem-cone-max} applied with $\zeta  \in (0,Q-\gamma)$. 

To obtain~\eqref{eqn-weight-compare}, we first use Lemma~\ref{lem-max-cell-diam} (with $q\rta 0$) to find that with probability $1- o_\ep^\infty(\ep)$, each cell $\eta([x-\ep,x])$ for $x\in\ep\BB Z$ which intersects $B_{1/2}(0)$ is contained in $\BB D$. The bound~\eqref{eqn-weight-compare} follows from this,~\eqref{eqn-max-shift}, and the second formula for $\omega_{\mcl G}^\ep$ in~\eqref{eqn-vertex-weight}. 
\end{proof}

\subsection{Euclidean displacement of the embedded walk}
\label{sec-euclidean-dist}

Recall the planar map $\rng M_{I^\ep} = M_{I^\ep} \setminus \mcl E(\bdy M_{I^\ep})$ from~\eqref{eqn-interior-graph}. 
As in Definition~\ref{def-walk}, for $\ep \in (0,1)$, let $X^{\mcl G^\ep}$ by a simple random walk on $\mcl G^\ep$ started from 0 and let $X^{\rng M_{I^\ep}}$ be a simple random walk on $\rng M_{I^\ep}$ started from $\BB v$. In this subsection we will apply Markov type theory to bound the Euclidean displacement of these walks under the embeddings defined in Section~\ref{sec-core-setup}. The main result of this subsection is the following proposition. 
We will apply the proposition to deduce that the embedded walk typically takes time at least $\ep^{-1+o_\ep(1)}$ to exit $\BB D$. This is done by taking both the parameters $\zeta$ and $\wh\zeta$ to  be small.

\begin{prop}  \label{prop-euclidean-dist}
For each $\zeta,\wh\zeta \in (0,1)$ with $2\zeta <\wh\zeta$, there exists $\alpha=\alpha(\zeta,\wh\zeta , \gamma)  >0$ such that for each $\ep \in (0,1)$, 
\eqb \label{eqn-sg-euclidean}
\BB P\left[  \max_{j \in [0,\ep^{-1+\wh \zeta} ]_{\BB Z}}  | \eta( X_j^{\mcl G^\ep}) |   \leq   \ep^\zeta   \right]  \geq 1 - O_\ep(\ep^\alpha)
\eqe 
and
\eqb \label{eqn-map-euclidean}
\BB P\left[   \max_{j \in [0,\ep^{-1+\wh\zeta}]_{\BB Z}} \left| \Phi^\ep\left(X_j^{\rng M_{I^\ep}} \right)  \right|  \leq \ep^\zeta \right]  \geq 1 - O_\ep(\ep^\alpha) .
\eqe 
at a rate depending only on $\zeta,\wh\zeta$, and $\gamma$.  
\end{prop}
 
To prove Proposition~\ref{prop-euclidean-dist}, we start with a bound for displacement with respect to the vertex-weighted graph distance (with the weights defined as in Section~\ref{sec-weight-def}), then show that this vertex-weighted graph distance is comparable to embedded Euclidean distance. 

\begin{lem} \label{lem-weight-dist}
Define the vertex weightings $\omega_{\mcl G}^\ep$ and $\omega_M^\ep$ as in~\eqref{eqn-vertex-weight} and~\eqref{eqn-map-weight}, respectively, and let
\eqb \label{eqn-weight-dist-def}
d_{\omega_{\mcl G}^\ep} := \op{dist}_{\omega_{\mcl G}^\ep}^{\mcl G^\ep}(\cdot,\cdot) \quad \text{and} \quad d_{\omega_M^\ep} := \op{dist}_{\omega_{M}^\ep}^{\rng M_{I^\ep}}(\cdot,\cdot)
\eqe 
be the associated weighted graph distances as in~\eqref{eqn-weighted-dist}. Then
\eqb \label{eqn-weight-dist}
\BB E\left[ \max_{j\in [0,n]_{\BB Z}} d_{\omega_{\mcl G}^\ep} \left( 0 , X_j^{\mcl G^\ep} \right)^2  \right]  \leq n \ep^{ 1 + o_\ep(1)} .
\eqe
Furthermore, if the exponent $K$ from~\eqref{eqn-coupling-interval} is chosen to be sufficiently large (depending only on $\gamma$) and we set $F^\ep := \{ \BB e \in \mcl E(M_{I^\ep} )  \setminus  \mcl E(\bdy M_{I^\ep} ) \}$, then $\BB P[F^\ep] \geq 1 -O_\ep(\ep^{100})$ and
\eqb \label{eqn-weight-dist-map}
\BB E\left[  \BB 1_{F^\ep}  \max_{j\in [0,n]_{\BB Z}} d_{\omega_M^\ep} \left( \BB v , X_j^{\rng M_{I^\ep} } \right)^2 \right]  \leq n \ep^{ 1 + o_\ep(1)}   . 
\eqe
\end{lem}
\begin{proof}
By Lemma~\ref{lem-unimodular-weight}, we can apply Corollary~\ref{cor:UnimodularMarkovType} and~\eqref{eq:reversibleDLP}, respectively, to the weighted graphs $(\mcl G^\ep  , \omega_{\mcl G}^\ep , 0)$ and $(\rng M_{I^\ep} , \omega_M^\ep , \BB v)$ to find that there is a universal constant $C >0$ such that for each $\ep \in (0,1)$ and each $n\in\BB N$, 
\eqb \label{eqn-mtype}
\BB E\left[ \max_{j\in [0,n]_{\BB Z}} d_{\omega_{\mcl G}^\ep} \left( 0 , X_j^{\mcl G^\ep} \right)^2 \deg^{\mathcal{G}^\ep}(0)  \right]  \leq n C^2  \BB E\left[ \omega_{\mcl G}^\ep(0)^2  \deg^{\mathcal{G}^\ep}(0)  \right]  
\eqe
 and
\eqb \label{eqn-mtype-map}
\BB E\left[ \max_{j\in [0,n]_{\BB Z}} d_{\omega_M^\ep} \left( \BB v , X_j^{\rng M_{I^\ep} } \right)^2 \,|\, F^\ep \right]  \leq n C^2  \BB E\left[ \omega_M^\ep(\BB v)^2 \,|\, F^\ep \right]  .
\eqe 
We note that $\rng M_{I^\ep}$ is not necessarily connected, but this is no problem since we can apply~\eqref{eq:reversibleDLP} in each connected component separately. 
The bound~\eqref{eqn-weight-dist} follows by combining~\eqref{eqn-mtype} with Proposition~\ref{prop-weight-moment}. 

To deduce~\eqref{eqn-weight-dist-map}, we need to get rid of the conditioning in~\eqref{eqn-mtype-map}. To accomplish this, we first note that since the index shift $\theta^\ep$ from~\eqref{eqn-coupling-interval} is uniform on $[0, \ep^{-K} ]_{\BB Z}$, the probability that $[-\ep^{-K/2} , \ep^{-K/2}] \subset I^\ep$ is $1-O_\ep(\ep^{K/2})$. By the proof of~\cite[Lemma 1.11]{ghs-map-dist} (see the discussion just after~\cite[Equation (1.24)]{ghs-map-dist}), if $K > 200$ is chosen to be sufficiently large (depending only on $\gamma$) then it holds with probability at least $1-O_\ep(\ep^{100})$ that each vertex of $M$ which neighbors the root edge $\BB e$ is contained in $\iota_{[-\ep^{-K/2} , \ep^{-K/2}]}\left( M_{[-\ep^{-K/2} , \ep^{-K/2}]} \setminus \bdy M_{[-\ep^{-K/2} , \ep^{-K/2}]} \right)$. If this is the case and $[-\ep^{-K/2} , \ep^{-K/2}] \subset I^\ep$ (which happens with probability $1-O_\ep(\ep^{100})$), then $F^\ep$ occurs. Hence $\BB P[F^\ep] \geq 1 - O_\ep(\ep^{100})$.  
Combining this with Proposition~\ref{prop-weight-moment} allows us to bound the right side of~\eqref{eqn-mtype-map} and thereby deduce~\eqref{eqn-weight-dist-map}.
\end{proof}

\begin{proof}[Proof of Proposition~\ref{prop-euclidean-dist}] 
We will prove the proposition by showing that Euclidean distances can be bounded above in terms of $\omega_{\mcl G}^\ep$- and $\omega_M^\ep$-weighted distances with high probability and applying Lemma~\ref{lem-weight-dist}. To do this we will first define an event which happens with probability at least $1-O_\ep(\ep^\alpha)$, then compare distances on this event. 
\medskip

\noindent\textit{Step 1: definition of a regularity event.}
Let $q := \frac{1}{(2+\gamma)^2}$ (any other $q \in \left(0 , \frac{2}{(2+\gamma)^2}\right)$ would do equally well). 
Also let $\zeta$ and $\wh\zeta$ be as in the statement of the lemma, set $\delta  := (\zeta \wedge (\wh\zeta -\zeta)) /100 $, and let $E^\ep = E^\ep(\zeta,\wh\zeta,q)$ be the event that the following is true.
\begin{enumerate}
\item We have \label{item-dist-event-compare}
\eqb \label{eqn-dist-event-compare}
\op{diam}\left( \bigcup_{y\in \mcl V\mcl B^{\mcl G^\ep}_{(\log \ep^{-1})^p}(x)} \eta([y-\ep,y]) \right)  \leq \ep^{-\delta} \omega_{\mcl G}^\ep(x)  ,\quad \forall x \in \mcl V\mcl G^\ep(B_{1/2}(0) ) .
\eqe
\item In the notation of~\eqref{eqn-weight-dist-def}, \label{item-dist-event-weight}
\eqb \label{eqn-dist-event-weight}
\max_{j\in [0,\ep^{-1+\wh\zeta}]_{\BB Z}} d_{\omega_{\mcl G}^\ep} \left( 0 , X_j^{\mcl G^\ep} \right) \leq \ep^{\zeta+2\delta} \quad\op{and}\quad
\max_{j\in [0,\ep^{-1+\wh\zeta} ]_{\BB Z}} d_{\omega_M^\ep} \left( \BB v , X_j^{\rng M_{I^\ep}} \right) \leq \ep^{\zeta+2\delta} .
\eqe
\item Each cell $\eta([x-\ep,x])$ for $x\in\mcl V\mcl G^\ep(B_{1/2}(0) )$ has Euclidean diameter at most $\ep^q$. \label{item-dist-event-diam}
\item The coupling conditions in Theorem~\ref{thm-map-count} are satisfied with $I^\ep$ as in~\eqref{eqn-coupling-interval}. \label{item-dist-event-coupling}
\end{enumerate}
By Proposition~\ref{prop-max-shift} applied with $S = \ep^{-\delta}$, condition~\ref{item-dist-event-compare} holds except on an event of probability decaying faster than some positive power of $\ep$. By Lemma~\ref{lem-weight-dist} (applied with $n = \lfloor \ep^{-1+\wh\zeta} \rfloor$) and since $\wh\zeta > \zeta+2\delta$, we can apply the Chebyshev inequality to get that the same is true for condition~\ref{item-dist-event-weight}. By Lemma~\ref{lem-max-cell-diam} and our choice of coupling, respectively, the same is also true for conditions~\ref{item-dist-event-diam} and~\ref{item-dist-event-coupling}. Therefore, we can find $\alpha >0$ as in the statement of the lemma such that
\eqbn
\BB P\left[ E^\ep \right] \geq 1 - O_\ep(\ep^\alpha) .
\eqen
\medskip

\noindent\textit{Step 2: comparison of distances.} Henceforth assume that $E^\ep$ occurs. We will show that the events in~\eqref{eqn-sg-euclidean} and~\eqref{eqn-map-euclidean} hold by comparing $d_{\omega_{\mcl G}^\ep}$- and $d_{\omega_M^\ep}$-distances to Euclidean distances. The required analysis is straightforward and elementary. 
Actually, we will give the argument only in the case of $d_{\omega_M^\ep}$---the argument for $d_{\omega_{\mcl G}^\ep}$ is nearly the same, but slightly simpler.  

Suppose $v \in \mcl V(M_{I^\ep} )$ is such that $\Phi^\ep(v) \in \BB D$ (eventually, we will take $v$ to be one of the vertices $X_j^{\rng M_{I^\ep} }$). Let $P^M : [0, |P^M|]_{\BB Z} \rta \mcl V(M)$ be a path in $\rng M_{I^\ep}$ from $\BB v$ to $v$. 
The image under the embedding $\Phi^\ep$ of the path $P^M$ is not necessarily contained in $B_{1/2}(0)$. Thus, in order to apply~\eqref{eqn-dist-event-compare} we let $i_*$ be the smallest $i\in [0,|P^M|]_{\BB Z}$ such that $\Phi^\ep(P^M(i_* + 1)) \not\in B_{1/2}(0)$, or $i_* := |P^M|$ if $\Phi^\ep(P^M)$ is entirely contained in $B_{1/2}(0)$. We also define $v_* := P^M(i_*)$. We observe that $|\Phi^\ep(v)| \leq 1$ (since $\Phi^\ep(v) \in\BB D$) and that by condition~\ref{item-dist-event-diam} in the definition of $E^\ep$, either $v_* = v$ or $|\Phi^\ep(v_*)| \geq 1/2 - o_\ep(1)$. In particular, for small enough $\ep \in (0,1)$ we have
\eqb \label{eqn-first-vertex}
|\Phi^\ep(v)| \leq 4 |\Phi^\ep(v_*)| .
\eqe 

We now concatenate the paths of condition~\ref{item-map-count-G} of Theorem~\ref{thm-map-count} (with $n = \lfloor \ep^{-K} \rfloor$) in $\mcl G^\ep$ between the pairs of vertices $(\phi^\ep( P^M(i-1)  ), \phi^\ep( P^M(i) ))$, which each have length at most $ (\log \ep^{-K})^{p_0}  $. Recalling that $p > p_0$, we obtain a path $P^{\mcl G^\ep}$ in $\mcl G^\ep$ from $0$ to $\phi^\ep( v_* )$ such that
\eqb \label{eqn-G-path-contained}
P^{\mcl G^\ep}\left([0,|P^{\mcl G^\ep}|]_{\BB Z} \right) \subset \bigcup_{i=1}^{ i_* } \mcl B_{(\log \ep^{-1})^p}^{\mcl G^\ep}(\phi^\ep(P^M(i)) . 
\eqe 
Since the cells corresponding to any two consecutive vertices hit by the path $P^{\mcl G^\ep}$ intersect, the Euclidean distance between the cells corresponding to the endpoints of $P^{\mcl G^\ep}$ is at most the diameter of the union of the cells corresponding to the vertices hit by $P^{\mcl G^\ep}$. Note that the starting point of $P^{\mcl G^\ep}$ is 0. We therefore infer from~\eqref{eqn-G-path-contained} that
\eqb \label{eqn-dist-to-sum}
|  \Phi^\ep(v_*) | \leq  \sum_{i=1}^{ i_* }  \op{diam}\left( \bigcup_{y\in \mcl V\mcl B_{(\log \ep^{-1})^p}^{\mcl G^\ep}(\phi^\ep(P^M(i)) }  \eta([y-\ep,y])  \right) .
\eqe 

By the definition of $i_*$ and~\eqref{eqn-coupling-function}, each of the vertices $\phi^\ep(P^M(i))$ for $i\in [0,i_*]_{\BB Z}$ satisfies $\eta(\phi^\ep(P^M(i)) ) = \Phi^\ep(P^M(i) ) \in B_{1/2}(0)$. Therefore, condition~\ref{item-dist-event-compare} in the definition of $E^\ep$ applied to the right side of~\eqref{eqn-dist-to-sum} shows that
\eqbn
| \Phi^\ep(v_*) | 
\leq \ep^{-\delta} \sum_{i=1}^{i_*} \omega_{\mcl G}^\ep\left( \phi^\ep(P^M(i)) \right) 
= \ep^{-\delta} \sum_{i=1}^{i_*} \omega_M^\ep\left( P^M(i)  \right) , 
\eqen
which is at most $\ep^{-\delta}$ times the $\omega_M^\ep$-length of $P^M$. Taking the infimum over all choices of paths $P^M$ from $\BB v$ to $v$ in $\rng M_{I^\ep}$ and recalling~\eqref{eqn-first-vertex} shows that
\eqb \label{eqn-vertex-dist}
 |\Phi^\ep(v)| \leq 4 |\Phi^\ep(v_*)| \leq \ep^{-\delta} d_{\omega_M^\ep} \left( \BB v , v \right) .
\eqe 

Let $j_*$ the smallest $j\in [0,\ep^{-1+\wh\zeta}]_{\BB Z}$ for which $|\Phi^\ep(X_j^{M,\ep} )| \geq 4\ep^{\zeta + \delta}$, or $j_* = \lfloor \ep^{-1+\wh\zeta} \rfloor$ if no such $j$ exists. Applying~\eqref{eqn-vertex-dist} with $v= X_{j_*}^{\rng M_{I^\ep}}$ and recalling condition~\ref{item-dist-event-weight} in the definition of $E^\ep$ gives
\eqbn
\max_{j\in [0,\ep^{-1+\wh\zeta} ]_{\BB Z}} d_{\omega_M^\ep}  \left( \BB v , X_j^{\rng M_{I^\ep} } \right) \leq 4 \ep^{\zeta +\delta} 
\eqen
which is smaller than $\ep^\zeta$ for small enough $\ep$, as required.
\end{proof}

\begin{proof}[Proof of Theorem~\ref{thm-euc-displacement}]
We first deduce the upper bound for the Euclidean displacement of the embedded walk $\eta(X^{\mcl G^1})$ from Proposition~\ref{prop-euclidean-dist} and a scaling argument. 
Fix $\zeta,\wh\zeta \in (0,1)$ with $2\zeta  <\wh\zeta$ and let $\alpha = \alpha( \zeta,\wh\zeta,\gamma)  > 0$ be as in Proposition~\ref{prop-euclidean-dist}. 
Also recall the radii $R_b$ for $b>0$ from~\eqref{eqn-mass-hit-time}. 
Taking $\ep = 1/n$ in Proposition~\ref{prop-euclidean-dist} and applying the scaling property~\eqref{eqn-cone-scale} of the $\gamma$-quantum cone, we find that with probability $1-O_n(n^{-\alpha})$, the embedded walk $\eta(X^{\mcl G^1})$ takes at least $n^{1 -\wh\zeta}$ units of time to exit the Euclidean ball $B_{R_n}(0)$. By Lemma~\ref{lem-cone-hit-tail}, it holds except on an event of probability decaying faster than some negative power of $n$ that $n^{\frac{1}{2-\gamma^2/2} - \wh\zeta} \leq R_n \leq n^{\frac{1}{2-\gamma^2/2} + \wh\zeta}$. Applying this estimate for dyadic values of $n$, using the Borel-Cantelli lemma, and sending $\wh\zeta \rta 0$ shows that a.s.\ 
\eqb \label{eqn-euc-displacement-upper}
\limsup_{n\rta\infty} \frac{\log \max_{1\leq j \leq n} |\eta(X_j^{\mathcal G^1})|}{\log n} \leq \frac{1}{2-\gamma^2/2} .
\eqe 

We now deduce the corresponding lower bound from~\cite[Proposition 3.4]{gm-spec-dim}. The argument is standard, and is very similar to various proofs in~\cite[Section 4.2]{gm-spec-dim}, so we will be terse. For $n\in\BB N$, let $T_n$ be first time $X^{\mcl G^1}$ hits $\mcl G^1(\bdy B_{R_n}(0))$. We will establish an upper bound for $\BB E[T_n \,|\, (h,\eta) ]$, where $h$ and $\eta$ are as in Section~\ref{sec-peanosphere}. By~\cite[Proposition 3.4]{gm-spec-dim}, applied with $\ep=1/n$, together with the scaling property of the $\gamma$-quantum cone, applied as above, there are constants $C,\alpha_0 > 0$, depending only on $\gamma$, such that with probability at least $1-O_n\left( (\log n)^{-\alpha_0} \right)$, the effective resistance from 0 to $\mcl G^1(\bdy B_{R_n}(0))$ in $\mcl G^1$ is at most $C\log n$, i.e., the Green's function of $X^{\mcl G^1}$ stopped at time $T_n$ satisfies
\eqb \label{eqn-green0}
\op{Gr}_{T_n}(0,0) \leq C \op{deg}^{\mcl G^1}(0) \log n .
\eqe 
By reversibility of the Green's function (see, e.g.,~\cite[Exercise 2.1]{lyons-peres}), if~\eqref{eqn-green0} holds then for each $x\in \mcl V\mcl G^1(B_{R_n}(0))$, 
\begin{align}
\op{Gr}_{T_n}(0,x)
&= \frac{\op{deg}^{\mcl G^1}(x )}{\op{deg}^{\mcl G^1}(0)} \op{Gr}_{T_n} (x,0 )\nonumber\\
&= \frac{\op{deg}^{\mcl G^1}(x )}{\op{deg}^{\mcl G^1}(0)} \ol{\BB P}  \left[ \text{walk started at $x$ hits $0$ before $\bdy B_{R_n}(0) $} \,|\, \mcl G^1 \right] \op{Gr}_{T_n}(0,0) \notag\\
&\leq C \op{deg}^{\mcl G^1}(x ) \log n   ,\label{eqn-g-off-diagonal}
\end{align}
i.e., the expected number of times that $X^{\mcl G^1}$ hits $x$ before time $T_n$ is at most $ C \op{deg}^{\mcl G^1}(x ) \log n $.
Since the degree of each vertex of $\mcl G^1$ has an exponential tail~\cite[Lemma 2.2]{gms-harmonic} and by~\cite[Lemma A.4]{ghs-dist-exponent} (which we use to compare a Euclidean ball to a segment of $\eta$), we find that with probability $1-o_n^\infty(n)$, $\op{deg}^{\mcl G^1}(x) \leq (\log n)^2$ for each $x\in \mcl G^1(B_{R_n}(0))$. Plugging this into~\eqref{eqn-g-off-diagonal} and summing over all $x \in \mcl V\mcl G^1(B_{R_n}(0))$ gives that with probability $1-O_n\left( (\log n)^{-\alpha_0} \right)$, 
\eqb \label{eqn-euc-exit-upper}
\BB E\bigl[ T_n \,|\, (h,\eta) \bigr] \leq C (\log n)^3 \# \mcl V \mcl G^1(B_{R_n}(0)) .
\eqe 
By~\cite[Lemma~A.3]{ghs-dist-exponent} and~\eqref{eqn-cone-scale}, for each $\zeta>0$ there is an $\alpha_1 = \alpha_1(\zeta,\gamma) > 0$ such that with probability at least $1-O_n(n^{-\alpha_1})$, we have $\# \mcl V\mcl G^1(B_{R_n}(0)) \leq n^{1+\zeta}$ (here we recall that the cells in $\mcl G^1$ have $\mu_h$-mass 1). Combining this with~\eqref{eqn-euc-exit-upper} and Markov's inequality, we get that $T_n \leq n^{1+2\zeta}$ except on an event of probability decaying faster than some positive power of $(\log n)^{-1}$.

Applying Lemma~\ref{lem-cone-hit-tail} (exactly as in the proof of~\eqref{eqn-euc-displacement-upper}) and the above estimate with $n = n_m = \exp(m^s)$ for $s  > 2/\alpha_0$ and $m\in\BB N$, and taking a union bound over all $m\in\BB N$, we obtain that a.s.\ 
\eqb \label{eqn-euc-displacement-lower}
\liminf_{m \rta\infty} \frac{\log \max_{1\leq j \leq n_m } |\eta(X_j^{\mathcal G^1})|}{ \log n_m } \geq \frac{1}{2-\gamma^2/2} .
\eqe 
Since $\lim_{m\rta\infty} \log n_m / \log n_{m+1} =1$ and $n\mapsto \max_{1\leq j \leq n  } |\eta(X_j^{\mathcal G^1})|$ is increasing, we infer that~\eqref{eqn-euc-displacement-lower} remains true if we replace $n_m$ with a general $n\in\BB N$ and take a liminf as $n\rta\infty$. 
Combining with~\eqref{eqn-euc-displacement-upper} concludes the proof. 
\end{proof}

\subsection{Comparing Euclidean distances to unweighted graph distances}
\label{sec-graph-dist}

To deduce Theorem~\ref{thm-walk-speed} from Proposition~\ref{prop-euclidean-dist}, we need to compare embedded Euclidean distances to (unweighted) graph distances. For this purpose it will be enough to consider $\mcl G^\ep$-graph distances since Theorem~\ref{thm-map-count} allows us to compare such distances to $\rng M_{I^\ep}$-graph distances. 
 We will use the following result, which is part of~\cite[Proposition 4.6]{dg-lqg-dim} (note that, in the notation of~\cite{dg-lqg-dim}, $D_{h,\eta}^\ep(z,w ; \BB D)$ denotes the graph distance in $\mcl G^\ep$ between the cells containing $z$ and $w$ along paths whose corresponding cells are contained in $\BB D$). 
 
\begin{prop}[\!\cite{dg-lqg-dim}] \label{prop-ball-compare}
Let $d_\gamma$ be as in~\eqref{eqn-dist-exponent}. For each $\zeta \in (0,1)$, there exists $\alpha=\alpha(\zeta, \gamma) > 0$ such that for each $\ep \in (0,1)$, 
\eqbn
\BB P\left[  \mcl G^\ep\left(B_{1/2}(0) \right) \subset \mcl B^{\mcl G^\ep}_{\ep^{-1/d_\gamma - \zeta}}(0) \right] \geq 1 - O_\ep(\ep^\alpha)  .
\eqen 
\end{prop}  

Our main result is an easy consequence of Propositions~\ref{prop-euclidean-dist} and~\ref{prop-ball-compare} together with the lower bound for the speed of the walk from~\cite{ghs-map-dist}. 

\begin{proof}[Proof of Theorem~\ref{thm-walk-speed}]
We will first prove~\eqref{eqn-walk-speed}. 
By Propositions~\ref{prop-euclidean-dist} and~\ref{prop-ball-compare}, for each $\delta \in (0,1)$ there exists $\alpha=\alpha(\delta ,\gamma) > 0$ such that with probability $1-O_\ep(\ep^\alpha)$, 
\eqb \label{eqn-walk-contain}
 \mcl G^\ep\left(B_{1/2}(0) \right) \subset \mcl B_{\ep^{-1/d_\gamma -\delta }}^{\mcl G^\ep}(0) .
\eqe
and
\eqb \label{eqn-use-map-euc}
 \max_{j \in [0,\ep^{-1+ \delta }]_{\BB Z}} \left| \eta \left(X_j^{\mcl G^\ep} \right)  \right|  \leq \frac12 \quad \text{and} \quad
 \max_{j \in [0,\ep^{-1+ \delta }]_{\BB Z}} \left| \Phi^\ep\left(X_j^{\rng M_{I^\ep} } \right)  \right|  \leq \frac12 .
\eqe 

To deduce~\eqref{eqn-walk-speed} from this in the case of the mated-CRT map, for a given $n\in\BB N$, we choose $\ep \in (0,1)$ such that $\ep^{-1 + \delta } = n$. If the events in~\eqref{eqn-walk-contain} and~\eqref{eqn-use-map-euc} hold for this choice of $\ep$ and $\delta \in (0,1)$ is chosen sufficiently small relative to $\zeta$, then 
\eqbn
X^{\mcl G^\ep}([0,n]_{\BB Z}) \subset  \mcl G^\ep\left(B_{1/2}(0) \right) \subset  \mcl B_{n^{1/d_\gamma +\zeta}}^{\mcl G^\ep}(0) 
\eqen
which gives~\eqref{eqn-walk-speed} with $(\mcl G  ,0)$ in place of $(M,\BB v)$ since the law of $\mcl G^\ep$ does not depend on $\ep$. 
  
We now prove~\eqref{eqn-walk-speed} when $(M,\BB v)$ is one of the other four random planar maps listed in Section~\ref{sec-main-results}. 
Recall the functions $\phi^\ep : \mcl V M_{I^\ep} \rta \mcl V\mcl G^\ep_{\ep I^\ep}$ and $\psi^\ep :  \mcl V\mcl G^\ep_{\ep I^\ep} \rta \mcl V M_{I^\ep}$ from~\eqref{eqn-coupling-function}.
If~\eqref{eqn-walk-contain} and~\eqref{eqn-use-map-euc} both hold (which happens with probability $1-O_\ep(\ep^\alpha)$) then since $\Phi^\ep = \eta\circ\phi^\ep$, 
\eqbn
\phi^\ep\left( X^{\rng M_{I^\ep}}([0,\ep^{-1+\delta}]_{\BB Z} ) \right) \subset \mcl G^\ep\left(B_{1/2}(0) \right) \subset \mcl B_{\ep^{-1/d_\gamma -\delta }}^{\mcl G^\ep}(0) .
\eqen
By our choice of coupling (in particular, condition~\ref{item-map-count-M} of Theorem~\ref{thm-map-count}), it holds with probability $1-o_\ep^\infty(\ep)$ that 
\eqbn
\psi^\ep\left( \mcl B_{\ep^{-1/d_\gamma -\delta }}^{\mcl G^\ep}(0) \right) \subset \mcl B_{\ep^{-1/d_\gamma - 2\delta } -1}^{ M_{I^\ep} }(\BB v)
\quad \text{which implies} \quad
\psi^\ep\left( \mcl B_{\ep^{-1/d_\gamma -\delta }}^{\mcl G^\ep}(0) \right) \subset \mcl B_{\ep^{-1/d_\gamma - 2\delta } }^{\rng M_{I^\ep} }(\BB v)
\eqen
Combining the preceding two inequalities with condition~\ref{item-map-count-close} of Theorem~\ref{thm-map-count} shows that
\eqb \label{eqn-interval-walk}
\BB P\left[ \max_{j \in [0,\ep^{-1+\delta}]_{\BB Z}} \op{dist}^{\rng M_{I^\ep}}\left(  \BB v , X^{\rng M_{I^\ep}}_j  \right) \leq \ep^{-1/d_\gamma - 3\delta} \right] \geq 1 - O_\ep(\ep^\alpha) .
\eqe

We will now choose the exponent $K$ from~\eqref{eqn-coupling-interval} large enough to allow us to compare $\rng M_{I^\ep}$ and $M$. By~\cite[Lemma 1.11]{ghs-map-dist}, if we choose $K $ sufficiently large, depending only on $\gamma$, then with probability at least $1-O_\ep(\ep)$ the ``almost inclusion" function $\iota_{I^\ep } : M_{I^\ep } \rta M$ restricts to a graph isomorphism from $\mcl B^{M_{I^\ep }}_{\ep^{-1 }+1}(\BB v) $ to $\mcl B^{M }_{\ep^{-1 }+1}(\BB v) $. Recalling that $\rng M_{I^\ep} = M_{I^\ep} \setminus \mcl E(\bdy M_{I^\ep})$, we see that for such a choice of $K$ it holds with probability $1-O_\ep(\ep)$ that $\iota_{I^\ep }  $ restricts to a graph isomorphism from $\mcl B^{\rng M_{ I^\ep }}_{\ep^{-1 }}(\BB v) $ to $\mcl B^{M }_{\ep^{-1 }}(\BB v) $. 
 

Since $X^{\rng M_{I^\ep}}$ cannot leave $\mcl B^{\rng M_{ I^\ep }}_{\ep^{-1 }}(\BB v) $ in fewer than $\ep^{-1}$ steps, we find that~\eqref{eqn-interval-walk} holds with $M$ in place of $\rng M_{I^\ep}$, for every $\ep \in (0,1)$. We then obtain~\eqref{eqn-walk-speed} for $(M,\BB v)$ by choosing $\ep \in (0,1)$ so that $\ep^{-1+\delta} = n$ and $\delta \in (0,1)$ small enough (depending only on $\zeta$ and $\gamma$) that $\ep^{-1/d_\gamma  -3\delta}  \leq n^{1/d_\gamma + \zeta}$. 

It remains only to prove~\eqref{eqn-walk-speed-a.s.}. It is immediate from~\eqref{eqn-walk-speed} and a union bound over dyadic scales that a.s.\ 
\eqbn
\limsup_{n\rta\infty} \frac{\log \max_{j \in [1,n]_{\BB Z}} X_j^M}{\log n} \leq \frac{1}{d_\gamma} .
\eqen
We now explain exactly how one extracts the corresponding lower bound from the results of~\cite{gm-spec-dim}. The estimates in this case are slightly more delicate since one only has polylogarithmic, rather than polynomial, bounds for probabilities. By~\cite[Theorem 1.7]{gm-spec-dim}, if we let $\sigma_r$ be the exit time of $X^M$ from the ball $\mcl B_r^M(\BB v)$, then we can find constants $\alpha , p  > 0$ such that with probability at least $1-O_r( (\log r)^{-\alpha})$, the \emph{conditional} expectation of $\sigma_r$ given $(M,\BB v)$ is at most $(\log r)^p \#\mcl V\mcl B_r^M(\BB v)$. By the Chebyshev inequality,  
\eqb \label{eqn-exit-upper}
\BB P\left[  \sigma_r \leq (\log r)^{p+\alpha/2} \#\mcl V\mcl B_r^M(\BB v) \right] \geq 1 - O_r\left(\frac{1}{(\log r)^{\alpha/2} } \right) .
\eqe 
Now fix a constant $s  > 2/\alpha$ and for $k\in\BB N$ let $r_k := \exp(k^s)$. By~\eqref{eqn-exit-upper} and the Borel-Cantelli lemma, a.s.\
\eqb \label{eqn-upper-subsequence} 
\sigma_{r_k} \leq (\log r_k)^{p + \alpha/2} \#\mcl V\mcl B_{r_k}^M(\BB v) ,\quad \text{for all large enough $k\in\BB N$.}
\eqe 
On the other hand,~\cite[Theorem 1.6]{dg-lqg-dim} shows that a.s.\ $ \#\mcl V\mcl B_r^M(\BB v) \leq r^{d_\gamma +o_r(1)}$ as $r\rta\infty$. By combining this with~\eqref{eqn-upper-subsequence}, we get that a.s.\ 
\eqbn
\limsup_{k\rta\infty} \frac{\log \sigma_{r_k}}{\log r_k} \leq d_\gamma .
\eqen
Since $\sigma_{r_k} \leq \sigma_r \leq \sigma_{r_{k+1}}$ for $r\in [r_k,r_{k+1}]$ and $(\log r_{k+1})/(\log r_k) \rta 1$ as $k\rta\infty$, this implies that a.s.\ $\limsup_{r \rta\infty} \log \sigma_r/\log r \leq d_\gamma$, which in turn implies the lower bound in~\eqref{eqn-walk-speed-a.s.}.
\end{proof}

\begin{remark} \label{remark-spec-dim}
Using the upper bound in Theorem~\ref{thm-walk-speed-uipt}, one can show that a.s.\ the conditional probability given $M$ that $X^M$ returns to its starting point after $n$ steps is at least $n^{-1+o_n(1)}$. This was originally established in~\cite[Theorem 1.7]{lee-conformal-growth} for a more general class of random planar maps and with a polylogarithmic error instead of an $n^{o_n(1)}$ error (see~\cite[Appendix A]{gm-spec-dim} for an explanation of why~\cite[Theorem 1.7]{lee-conformal-growth} applies to planar maps with multiple edges and/or self-loops allowed). 
To obtain an alternative proof of the return probability lower bound, one uses the upper bound in Theorem~\ref{thm-walk-speed-uipt} together with a standard calculation using reversibility and H\"older's inequality (see~\cite[Lemma 4.7]{gm-spec-dim} or the proof of~\cite[Corollary 15]{benjamini-curien-uipq-walk}).
This lower bound can be combined with the corresponding upper bound~\cite[Theorem 1.5]{gm-spec-dim} to show that the a.s.\ spectral dimension of $M$ is 2, i.e., the return probability after $n$ steps is a.s.\ $n^{-1+o_n(1)}$. 
\end{remark}
 
\section{Some technical estimates}
\label{sec-estimates}
 
In this section we complete the proofs of some technical estimates which are stated in Section~\ref{sec-core-argument}.

\subsection{Proof of Proposition~\ref{prop-weight-moment}}
\label{sec-log-ball-moment}

Proposition~\ref{prop-weight-moment} is an immediate consequence of the following estimate for $\mcl G^\ep$.  

\begin{lem} \label{lem-log-ball-moment}
Let $h$ be as in Section~\ref{sec-core-setup}, so that $h$ is a circle average embedding of a $\gamma$-quantum cone.
For each $p , q  , \zeta > 0$, and each $\ep \in (0,1)$,
\eqb \label{eqn-log-ball-moment}
\BB E\left[ \left(  \min\left\{ \ep^{-q} , \op{diam}\left(  \bigcup_{x\in \mcl V\mcl B_{(\log \ep^{-1})^p }^{\mcl G^\ep}(0)} \eta([x-\ep,x]) \right) \right\} \right)^{2 +\zeta } \right] \leq \ep^{1 + o_\zeta(1)  +o_\ep(1)}
\eqe 
with the rate of the $o_\zeta(1)$ depending only on $p,q,$ and $\gamma$ and the rate of the $o_\ep(1)$ depending only on $p$, $q$, $\gamma$, and $\zeta$. 
\end{lem}

\begin{proof}[Proof of Proposition~\ref{prop-weight-moment}, assuming Lemma~\ref{lem-log-ball-moment}]
By Lemma~\ref{lem-log-ball-moment} and the definition~\eqref{eqn-vertex-weight} of $\omega_{\mcl G}^\ep$, we have $\BB E[ \omega_{\mcl G}^\ep(0)^{2+\zeta} ]\leq \ep^{ 1+ o_\zeta(1) + o_\ep(1)}$ for each $\zeta > 0$. By~\cite[Lemma 2.2]{gms-harmonic}, the law of $\op{deg}^{\mcl G^\ep}(0)$ does not depend on $\ep$ and has an exponential tail. Hence we can apply H\"older's inequality, then send $\zeta \rta 0$, to get that $\BB E[ \omega_{\mcl G}^\ep(0)^2 \op{deg}^{\mcl G^\ep}(0) ]\leq \ep^{ 1+   o_\ep(1)}$. By the definition~\eqref{eqn-map-weight} of $\omega_M^\ep$, we have $\omega_M^\ep(\BB v) = \omega_{\mcl G}^\ep(0)$, and hence also $\BB E\left[ \omega_M^\ep(\BB v)^2 \right] \leq \ep^{ 1+o_\ep(1)}$. 
\end{proof}

It remains to prove Lemma~\ref{lem-log-ball-moment}. We will first apply Lemma~\ref{lem-max-cell-diam} to lower-bound the minimal number of cells in a path from 0 to $\mcl G^\ep(\bdy B_r(0))$ for fixed $r\in (0,1)$; then apply the scaling property of the $\gamma$-quantum cone discussed in Section~\ref{sec-lqg-prelim} to transfer from a macroscopic ball to a ball of radius $\ep^{1+o_\ep(1)}$. 

\begin{lem} \label{lem-ball-contained}
Suppose $h$ is a circle average embedding of a $\gamma$-quantum cone. For each fixed $r \in (0,1)$, each $p > 0$, and each $\ep \in (0,1)$,  
\eqb \label{eqn-ball-contained}
\BB P\left[ \bigcup_{x\in \mcl V\mcl B_{(\log \ep^{-1})^p}^{\mcl G^\ep}\left( 0  \right)} \eta([x-\ep,x])  \subset    B_r(0 )  \right] \geq 1 - o_\ep^\infty(\ep) 
\eqe 
at a rate depending only on $r$ and $p$. 
\end{lem}
\begin{proof}
Fix $r' \in (r ,1)$. By Lemma~\ref{lem-max-cell-diam}, for $q \in \left( 0 , \frac{2}{(2+\gamma)^2} \right)$, there is an explicit exponent $\alpha(q) = \alpha(q,\gamma) > 0$ such that $\alpha(q) \rta \infty$ as $q\rta 0$ and with probability at least $1- \ep^{\alpha(q) + o_\ep(1)}$, 
\eqbn
 \max_{x\in \mcl V\mcl G^\ep( B_{r'}(0) ) } \op{diam}\left( \eta([x-\ep,x]) \right) \leq \ep^{q }  .
\eqen
If this is the case, then each path in $\mcl G^\ep$ from 0 to $\mcl V\mcl G^\ep(\bdy B_r(0))$ must have length at least $c \ep^{-q}$ for a constant $c = c(r) > 0$. This implies that for small enough $\ep  \in (0,1)$ (how small depends only on $p,q,r,r'$), 
\eqbn
\bigcup_{x\in \mcl V\mcl B_{(\log \ep^{-1})^p}^{\mcl G^\ep}\left( 0  \right)} \eta([x-\ep,x]) \subset \bigcup_{x\in \mcl V\mcl B_{c\ep^{-q}}^{\mcl G^\ep}\left( 0  \right)} \eta([x-\ep,x])  \subset B_r(0). 
\eqen
Therefore,
\eqbn
\BB P\left[\bigcup_{x\in \mcl V\mcl B_{(\log \ep^{-1})^p}^{\mcl G^\ep}\left( 0  \right)} \eta([x-\ep,x])  \subset B_r(0) \right] \geq   1 -  \ep^{ \alpha(q) + o_\ep(1) }  ,\quad \forall q \in \left(0 , \frac{2}{(2+\gamma)^2} \right)  .
\eqen 
Since $\alpha(q) \rta \infty$ as $q\rta 0$, this implies~\eqref{eqn-ball-contained}. 
\end{proof}

\begin{proof}[Proof of Lemma~\ref{lem-log-ball-moment}]
Let $R_b$ for $b >0$ be as in~\eqref{eqn-mass-hit-time}, so that the field/curve pair $(h^b,\eta^b)$ defined by
\eqbn
h^b(\cdot) = h(R_b\cdot) +Q\log R_b - \frac{1}{\gamma} \log b \quad \op{and} \quad \eta^b  := R_b^{-1} \eta 
\eqen
agrees in law with $(h,\eta)$. Now fix $\delta \in (0,1)$ (which we will eventually send to 0) and take $b = \ep^{1-\delta}$. If we let $\mcl G^{\ep^{1-\delta} ,\ep^{\delta}}$ be the mated-CRT map with cell size $\ep^{\delta}$ associated with $(h^{\ep^{1-\delta}} , \eta^{\ep^{1-\delta}})$, then the mapping $z\mapsto R_{\ep^{1-\delta}}  z$ takes the cells of $\mcl G^{\ep^{1-\delta},\ep^{\delta}} $ bijectively to the cells of $\mcl G^\ep $ and induces an isomorphism of these graphs. Therefore, Lemma~\ref{lem-ball-contained} applied with $\ep^{\delta}$ in place of $\ep$, $p$ replaced by a slightly larger parameter, and $(h^{\ep^{1-\delta}} , \eta^{\ep^{1-\delta}})$ in place of $(h,\eta)$ shows that with probability $1-o_\ep^\infty(\ep)$, 
\eqb \label{eqn-log-ball-contained}
\bigcup_{x\in \mcl V\mcl B_{(\log \ep^{-1})^p}^{\mcl G^\ep}\left( 0  \right)} \eta([x-\ep,x])  \subset B_{r R_{\ep^{1-\delta}}}(0)
  \quad \text{and so} \quad
  \op{diam}\left( \bigcup_{x\in \mcl V\mcl B_{(\log \ep^{-1})^p }^{\mcl G^\ep}(0)} \eta([x-\ep,x]) \right)  \leq  r R_{\ep^{1-\delta}} .
\eqe 

We now estimate $\BB E\left[ R_{\ep^{1-\delta}}^{2+\zeta} \right]$ for a given choice of $\zeta>0$. The discussion in Section~\ref{sec-cone-prelim} shows that $-\log R_{\ep^{1-\delta}}$ has the same law as the first time $t > 0$ that a standard linear Brownian motion with negative linear drift $-(Q-\gamma) t$ hits $\frac{1}{\gamma} \log \ep^{1-\delta}$. Let $W_t = B_t - (Q-\gamma) t$ be such a drifted Brownian motion and let $T_\ep$ be this hitting time. We seek to estimate $\BB E[\exp(-(2+\zeta) T_\ep)]$. We will do this via a similar argument as in~\cite[Section 4.1]{shef-kpz}.

For each $\beta \in \BB R$, the process $t\mapsto \exp(\beta B_t - \beta^2 t/2)$ is a martingale. Using the optional stopping theorem and the fact that $B_{T_\ep} =   (Q-\gamma) T_\ep    + \frac{1}{\gamma} \log \ep^{1-\delta}$, we get
\eqb \label{eqn-log-ball-R}
\BB E\left[ \exp\left( \left( \beta (Q-\gamma) - \frac{\beta^2}{2} \right) T_\ep \right) \right] = \ep^{-\frac{\beta}{\gamma}(1-\delta)} .
\eqe 
If we choose $\beta = - ( 4 -  \gamma^2 - \sqrt{16 + \gamma^4 + 8 \gamma^2 (1 + \zeta)})/(2 \gamma)$, then the coefficient on $T_\ep$ on the left is equal to $-2-\zeta$. Moreover, $\beta =  -\gamma + o_\zeta(1)$ (with the rate of the $o_\zeta(1)$ depending only on $p,q,\gamma$) so the right side of~\eqref{eqn-log-ball-R} is equal to $\ep^{ -(1-\delta)(1+o_\zeta(1)) }$.  Therefore,
\eqb \label{eqn-radius-moment}
\BB E\left[ R_{\ep^{1-\delta}}^{2+\zeta}  \right]  = \ep^{ -(1-\delta)(1+o_\zeta(1)) }.
\eqe 
Combining~\eqref{eqn-log-ball-contained} and~\eqref{eqn-radius-moment} and sending $\delta\rta 0$ yields~\eqref{eqn-log-ball-moment}. Note that the truncation by $\ep^{-q}$ is needed since~\eqref{eqn-log-ball-contained} only holds with probability $1-o_\ep^\infty(\ep)$, not with probability 1.  
\end{proof}

\subsection{Proof of Lemma~\ref{lem-cone-max}}
\label{sec-cone-max}

To prove Lemma~\ref{lem-cone-max}, we will first apply standard estimates for Gaussian processes to estimate the circle average process of a whole-plane GFF (Lemma~\ref{lem-gff-avg}). We will then transfer this to estimates for the restriction to $\BB D$ of the field associated with a $\gamma$-quantum cone (Lemma~\ref{lem-cone-avg}) and finally use the scaling property~\eqref{eqn-cone-scale} to conclude.

\begin{lem} \label{lem-gff-avg}
Let $\wh h$ be a whole-plane GFF normalized so that its circle average over $\bdy\BB D$ is 0. 
Let $\{\wh h_r(z) : r > 0, z\in\BB C\}$ be a continuous modification of the circle average process of $\wh h$ (which exists by~\cite[Proposition 3.1]{shef-kpz}).
For each $\zeta \in (0,1)$ and $\delta\in(0,1)$,  
\eqb \label{eqn-gff-avg}
\BB P\left[ |\wh h_r(z)| \leq (\log  \delta^{-1} )^{1/2+\zeta} ,\: \forall r \in [\delta , 1] ,\: \forall z\in B_r(0)  \right] \geq 1 - o_\delta^\infty(\delta) 
\eqe  
at a rate depending only on $\zeta$.  
\end{lem}
\begin{proof}
Since the law of $\wh h$ is scale invariant, modulo additive constant, and the difference between two circle averages does not depend on the choice of additive constant for the field, we find that for each $s > 0$, 
\eqb \label{eqn-gff-avg-compare}
\left\{ \wh h_r(z) - \wh h_s(0) : r\in [s,2s] ,\, z\in B_r(0) \right\} \eqD \left\{ \wh h_r(z)  : r\in [1,2] , \, z\in B_r(0)    \right\}  .
\eqe  
The process on the right side of~\eqref{eqn-gff-avg-compare} is centered Gaussian with variances bounded above by a universal constant. 
Moreover, $(r,z) \mapsto \wh h_r(z)$ is continuous on $(0,\infty) \times \BB C$ so the supremum of $\wh h_r(z)$ over $r \in [1,2]$ and $z\in B_r(0)$ is a.s.\ finite.
We may therefore apply the  Borell-TIS inequality~\cite{borell-tis1,borell-tis2} (see, e.g.,~\cite[Theorem 2.1.1]{adler-taylor-fields}) to this process to get that for $s \geq 1$, 
\eqbn
\BB P\left[\max \left\{|\wh h_r(z) - \wh h_s(0) | : r\in [s,2s] ,\, z\in B_r(0) \right\} > t \right] \leq c_0 e^{-c_1 t^2} ,\quad \forall t \geq 1
\eqen
for universal constants $c_0,c_1 >0$ (note that we absorbed the expectation of the maximum, which is finite by the Borell-TIS inequality, into $c_0$ and $c_1$). We now take a union bound over dyadic scales to find that, with probability at least $1-2 c_0 \log_2 \delta^{-1} e^{-c_1 (\log \delta^{-1})^{1+2\xi}}=1-o_\delta^\infty(\delta)$, 
\allb \label{eqn-gff-avg-scales}
  \max \left\{|\wh h_r(z) - \wh h_{2^{-k}}(0) | : r \in [2^{-k} , 2^{-k+1}] ,\, z\in B_r(0) \right\}  \leq   \frac12 \left(\log \delta^{-1}\right)^{ 1/2+\zeta }  ,\quad \forall k \in  \left[ 0 ,    \lceil \log_2 \delta^{-1} \rceil \right]_{\BB Z}    .
\alle
Since $\wh h_{2^{-k}}(0)$ is centered Gaussian with variance $\log 2^k$~\cite[Section 3.1]{shef-kpz}, for each $\zeta\in(0,1)$ it holds that
\eqb \label{eqn-gff-avg-gaussian} 
\BB P\left[  |\wh h_{2^{-k}}(0)| \leq  \frac12 (\log \delta^{-1})^{1/2 + \zeta}  ,\:   \forall k \in  \left[ 0 ,   \lceil \log_2 \delta^{-1} \rceil  \right]_{\BB Z}   \right] \geq 1 -o_\delta^\infty(\delta) .
\eqe 
Combining~\eqref{eqn-gff-avg-scales} and~\eqref{eqn-gff-avg-gaussian} yields~\eqref{eqn-gff-avg}. 
\end{proof}

Recall the random radii $R_b$ for $b > 0$ associated with a $\gamma$-quantum cone from~\eqref{eqn-mass-hit-time}, which are chosen so that typically $\mu_h(B_{R_b}(0)) \asymp b$. 

\begin{lem} \label{lem-cone-avg}
Let $h$ be a circle average embedding of a $\gamma$-quantum cone.
There exists $\alpha=\alpha( \gamma)>0$ such that for each $\zeta \in (0,1)$ and each $\delta \in (0,1)$, 
\eqb  \label{eqn-cone-avg}
\BB P\left[ |h_r(z) + \gamma \log r| \leq  ( \log \delta^{-1} )^{1/2+\zeta} ,\:  \forall r \in [R_{1/e} , 1 ] ,\: \forall z \in B_{r/2}(0) \right] \geq 1- O_\delta(\delta^\alpha) .
\eqe 
\end{lem}
\begin{proof}
By Lemma~\ref{lem-cone-hit-tail}, there exists $\alpha=\alpha(\gamma) > 0$ such that  
\eqb \label{eqn-use-cone-hit-tail}
\BB P\left[ R_{1/e} \geq  \delta \right] \geq 1 - O_\delta(\delta^\alpha)  \quad \op{as} \quad \delta \rta 0.
\eqe 
Recall from Section~\ref{sec-cone-prelim} that the restriction of the field $\rng h := h + \gamma \log |\cdot|$ to $\BB D$ agrees in law with the corresponding restriction of a whole-plane GFF normalized so that its circle average over $\bdy\BB D$ is zero. Hence~\eqref{eqn-use-cone-hit-tail} combined with Lemma~\ref{lem-gff-avg} implies that 
\eqbn
\BB P\left[ |\rng h_r(z) | \leq  ( \log \delta^{-1} )^{1/2+\zeta} :  \forall r \in [R_{1/e} , 1 ] ,\: \forall z \in B_r(0) \right] \geq 1- O_\delta(\delta^\alpha) .
\eqen
We now conclude by noting that $ \rng h_r(z) - h_r(z) $ is the average of $\gamma\log |w|$ over $\bdy B_r(z)$, and if $z \in B_{r/2}(0)$ and $w\in B_r(z)$ then $| \gamma \log |w| - \gamma \log r|$ is at most a constant depending only on $\gamma$, so $|\rng h_r(z) - h_r(z) - \gamma \log r|$ is uniformly bounded above.
\end{proof}

\begin{proof}[Proof of Lemma~\ref{lem-cone-max}]
Recall the radii $R_b$ for $b>0$ defined in~\eqref{eqn-mass-hit-time}. We will prove the lemma by applying Lemma~\ref{lem-cone-avg} at the scales $R_{e^k}$ for $k \in \BB N_0$. 
By the scaling property~\eqref{eqn-cone-scale} of the $\gamma$-quantum cone, we can apply Lemma~\ref{lem-cone-avg} with $\delta = e^k$ and the field $h^{e^k} := h(R_{e^k}\cdot) + Q\log R_{e^k} - \gamma^{-1} k \eqD h$ in place of $h$ to find that there is an $\alpha_0 = \alpha_0(\gamma) > 0$ such that for each $k\in\BB N $, it holds with probability $1- O_k(e^{-\alpha_0 k})$ that 
\eqb \label{eqn-use-cone-avg0}
  \left| h_r^{e^k}(z)  + \gamma \log r   \right| \leq  k^{1/2+\zeta}    ,\:  \forall r \in [R_{e^{k-1}} /R_{e^k} , 1 ] ,\: \forall z \in B_{r/2}(0) .
\eqe  
Here we note that $R_{e^{k-1}}/R_{e^{k}}$ is determined by $h^{e^k}$ in the same manner that $R_{1/e}$ is determined by $h$.
 If $ r \in [R_{e^{k-1}} /R_{e^k} , 1 ]$ and $z\in B_{r/2}(0)$, then with $r' :=   R_{e^k} r \in [R_{e^{k-1}} , R_{e^k}]$ and $z' := R_{e^k} z \in B_{R_{e^k} r/2}(0)$, we have
\alb
\left| h_r^{e^k}(z)  + \gamma \log r   \right|
= \left| h_{r'}(z')  + \gamma \log r'   + (Q-\gamma) \log R_{e^k} - \frac{1}{\gamma} k  \right| .
\ale 
Therefore, removing the primes to lighten notation and noting that $B_{r/2}(0) \subset B_{R_{e^k}r/2}(0)$ for $k\geq 0$, we see that~\eqref{eqn-use-cone-avg0} implies that 
\eqb  \label{eqn-use-cone-avg}
 \left| h_r(z)  + \gamma \log r  + (Q-\gamma) \log R_{e^k} - \frac{1}{\gamma} k  \right| \leq  k^{1/2+\zeta} ,\:  \forall r \in [R_{e^{k-1}} , R_{e^k} ] ,\: \forall z \in B_{r/2}(0)  .
\eqe 
By Lemma~\ref{lem-cone-hit-tail} (applied with $C = e^{\zeta k}$), there is a constant $\alpha_1 = \alpha_1(\zeta,\gamma) \in (0,\alpha_0]$ such that for each $k\in\BB N$,  
\eqbn
\BB P\left[  \exp\left( \left( \frac{k}{\gamma (Q-\gamma)}  - \zeta \right) k \right)   \leq  R_{e^k} \leq   \exp\left( \left( \frac{1}{\gamma (Q-\gamma)} + \zeta \right) k   \right) \right] 
\geq 1  - O_k(e^{-\alpha_1 k})  .
\eqen
Using this to estimate the term $(Q-\gamma) \log R_{e^k} - \gamma^{-1} k$ in~\eqref{eqn-use-cone-avg} gives that for a $\gamma$-dependent constant $C>0$, 
\eqbn
\BB P\left[ \left| h_r(z)  + \gamma \log r   \right| \leq  C \zeta k ,\:  \forall r \in [R_{e^{k-1}} , R_{e^k} ] ,\: \forall z \in B_{r/2}(0) \right] 
\geq 1- O_k(e^{-\alpha_1 k}) .
\eqen
If we are given $k_0 \in \BB N$, we can sum over all $k\geq k_0+1$ to get that for a possibly larger constant $C$, still depending only on $\gamma$,
\eqb  \label{eqn-outer-avg0}
\BB P\left[ \left| h_r(z)  + \gamma \log r   \right| \leq  C \zeta \log  r ,\:  \forall r \geq R_{e^{k_0}} ,\: \forall z \in B_{r/2}(0) \right] 
\geq 1- O_{k_0}(e^{-\alpha_1 k_0}) .
\eqe 
By Lemma~\ref{lem-cone-hit-tail}, there exists $\alpha_2=\alpha_2(\gamma) > 0$ such that for each $S \geq 1$ we can find $k_0 = k_0(S, \gamma) \in \BB N$ such that $e^{k_0}$ is at least some positive power of $S$ and $\BB P[ R_{e^{k_0}} \leq S ] \geq 1-O_S(S^{-\alpha_2})$. Since $\zeta \in (0,1)$ is arbitrary, we can combine this with~\eqref{eqn-outer-avg0} to get~\eqref{eqn-cone-max}. 
\end{proof}

\setstretch{1}

\bibliography{cibiblong,cibib,tomsrefs}
\bibliographystyle{hmralphaabbrv}

\end{document}